\documentclass[12pt,a4paper,final]{iopart}

\usepackage{iopams}
\usepackage{graphicx}
\usepackage[breaklinks=true,colorlinks=true,linkcolor=blue,urlcolor=blue,citecolor=blue]{hyperref}

\usepackage{tiaStyleIOP}

\usepackage{algorithmic}
\ifpdf%
  \DeclareGraphicsExtensions{.eps,.pdf,.png,.jpg}
\else
  \DeclareGraphicsExtensions{.eps}
\fi
\usepackage{amsopn}
\usepackage{booktabs}

\newcommand{\cref}[1]{\ref{#1}}
\newcommand{\arras}{\,\,\stackrel{{\rm a.s.}}{\longrightarrow}\,\,}
\expandafter\let\csname equation*\endcsname\relax
  \expandafter\let\csname endequation*\endcsname\relax

\RequirePackage{amsmath}
\usepackage{mathtools}
\DeclareMathOperator*{\argmin}{arg\,min}                   

\usepackage{xspace}

\newcommand{\rrls}{\texttt{rrls}\xspace}
\newcommand{\sTik}{\texttt{sTik}\xspace}
\newcommand{\slimTik}{\texttt{slimTik}\xspace}
\usepackage{multicol}

\begin{document}

\title[Sampled Tikhonov Regularization for Large Linear Inverse Problems]{Sampled Tikhonov Regularization for Large Linear Inverse Problems}

\author[cor1]{J. Tanner Slagel}
\address{Department of Mathematics, Virginia Tech, Blacksburg, VA}
\ead{slagelj@vt.edu}

\author{Julianne Chung}
\address{Department of Mathematics, Computational Modeling and Data Analytics Division, Academy of Integrated Science, Virginia Tech, Blacksburg, VA}
\ead{jmchung@vt.edu}

\author{Matthias Chung}
\address{Department of Mathematics, Computational Modeling and Data Analytics Division, Academy of Integrated Science, Virginia Tech, Blacksburg, VA}
\ead{mcchung@vt.edu}

\author{David Kozak}
\address{Applied Mathematics and Statistics, Colorado School of Mines, Golden, CO}
\ead{dkozak@mines.edu}

\author{Luis Tenorio}
\address{Applied Mathematics and Statistics, Colorado School of Mines, Golden, CO}
\ead{ltenorio@mines.edu}

\begin{abstract}
  In this paper, we investigate iterative methods that are based on sampling of the data for computing Tikhonov-regularized solutions.  We focus on very large inverse problems where access to the entire data set is not possible all at once (e.g., for problems with streaming or massive datasets).  Row-access methods provide an ideal framework for solving such problems, since they only require access to ``blocks'' of the data at any given time. However, when using these iterative sampling methods to solve inverse problems, the main challenges include a proper choice of the regularization parameter, appropriate sampling strategies, and a convergence analysis. To address these challenges, we first describe a family of sampled iterative methods that can incorporate data as they become available (e.g., randomly sampled).  We consider two sampled iterative methods, where the iterates can be characterized as solutions to a sequence of approximate Tikhonov problems. The first method requires the regularization parameter to be fixed a priori and converges asymptotically to an unregularized solution for randomly sampled data.  This is undesirable for inverse problems.  Thus, we focus on the second method where the main benefits are that the regularization parameter can be updated during the iterative process and the iterates converge asymptotically to a Tikhonov-regularized solution. We describe adaptive approaches to update the regularization parameter that are based on sampled residuals, and we describe a limited-memory variant for larger problems.  Numerical examples, including a large-scale super-resolution imaging example, demonstrate the potential for these methods.
\end{abstract}

\section{Introduction} \label{sec:Intro}
There have been significant developments in variational methods for solving large inverse problems \cite{ghattas13}. However, with recent advances in imaging technologies and new applications to computer vision and machine learning, datasets are becoming so large that existing methods, which often follow an ``all-at-once'' approach for processing the data, are no longer feasible.  Instead, we consider randomized or sampling methods where only ``blocks'' of the data are required at a given time.  Such methods are ideal for streaming problems, where data are generated or collected during the process of solution computation.

In this paper, we focus on linear inverse problems of the form,
\begin{equation*}
  \bfb = \bfA\bfx_{\rm true} + \bfepsilon,
\end{equation*}
where $\bfx_{\rm true} \in \bbR^n$ contains the desired, unknown parameters, $\bfA \in \bbR^{m \times n}$ models the data acquisition process, $\bfb \in \bbR^m$ contains the observed data (which may be streaming), and $\bfepsilon\in \bbR^m$ represents noise or errors in the data. We assume that $\bfepsilon$ has mean zero and a finite second moment.  In the generic setup, the goal of the inverse problem is to estimate $\bfx_{\rm true}$, given a model $\bfA$ and observations $\bfb$. Typically, the matrix $\bfA$ represents a discrete and linear version of a given model stemming for instance from a discretized PDE network, integral equation, or regression model \cite{hansen2006deblurring, mueller2012linear}.  For the problems of interest, $m$ and $n$ may be so large that accessing and/or storing all rows of $\bfA$ at once is infeasible.

In this work we consider \emph{ill-posed} inverse problems where regularization is required to compute reasonable solutions.
Here, we focus on solving the Tikhonov-regularized problem,
\begin{equation}\label{eq:ip}
  \min_{\bfx} f(\bfx) = \norm[2]{\bfA\bfx - \bfb}^2 + \lambda \norm[2]{\bfL\bfx}^2,
\end{equation}
where $\lambda >0$ is the regularization parameter, and for simplicity we assume that $\bfL$ has full column rank.
For the scenario where all of $\bfb$ and $\bfA$ are available or can be accessed at once (e.g., via matrix-vector multiplication with $\bfA$), the Tikhonov solution,
\begin{equation*}
  \bfx(\lambda) = (\bfA\t \bfA + \lambda\bfL\t\bfL)^{-1}\bfA\t \bfb\,,
\end{equation*}
can be computed using a plethora of existing iterative methods (e.g., Krylov or other optimization methods \cite{hansen2010discrete,kaipio2006statistical}). Note that $\bfx(0)$ is the unregularized solution which is defined if $\bfA$ has full column rank.

For the problems of interest, we consider sampled iterative methods of the form,
\begin{equation}\label{eq:iter}
  \bfx_k = \bfx_{k-1} - \bfB_k \bfg_k(\bfx_{k-1}), \qquad k \in \bbN,
\end{equation}
where $\bfx_0$ is an initial iterate, $\bfg_k(\bfx_{k-1})$ is a vector (carrying gradient information of the least squares problem), matrix $\bfB_k\in \bbR^{n \times n}$ is updated at each iteration (carrying curvature information of the least squares problem). A learning rate or line search parameter is not required in this case and is set to its ``natural'' value of 1, \cite{bottou2018optimization}.
Specific choices for $\bfB_k$ and $\bfg_k$ will be described in \cref{sec:rrls}, with connections to other known stochastic approximation methods described in \cref{sec:SO}.

Note that iterative methods of this form typically stem from nonlinear optimization problems where $\bfB_k$ is an approximation to the inverse Hessian and contains curvature information and $\bfg_k$ is the gradient at the current iterate $\bfx_{k-1}$ \cite{Nocedal2006}. Such methods would take only one step to converge for the linear problem (e.g., take $\bfx_0=\bf0$, $\bfB_1^{-1} = \bfA\t\bfA+\lambda\bfL\t\bfL$, and $\bfg_1 =\bfA\t\bfb$).  However, this is not possible if $m$ and $n$ are so large that not all information is available at a certain time or fits into computer memory.  Furthermore, determining a suitable choice of $\lambda$
can be computationally infeasible in such settings, and the information available, i.e., $\bfB_k$ and $\bfg_k$, may be subject to noise or other uncertainties. Thus, we consider nonlinear methods of the form \eqref{eq:iter} for Tikhonov regularization with massive data, where the main benefits are that (i) the data is sampled (e.g., randomly) or streamed, (ii) the regularization parameter can be adapted, and (iii) the methods converge asymptotically and in one epoch to a Tikhonov-regularized solution. Sophisticated regularization parameter selection methods are well-established if the full system is available (for example, see \cite{kilmer2001choosing,renaut2017hybrid}); however, the ability to update the regularization parameter within iterative methods of the form~\eqref{eq:iter} while also ensuring convergence of iterates to a regularized solution is, to the best of our knowledge, an unresolved problem.

\paragraph{Problem Formulation.}
In the following, we describe a mathematical formulation of the problem that allows us to solve~\eqref{eq:ip} in situations where
samples of $\bfA$ and $\bfb$ become available over time.
Such scenarios are common in medical imaging, e.g. in tomography where data is being processed as it is being collected \cite{andersen2014generalized}, and in astronomy, e.g. in super-resolution imaging where a high-resolution image is constructed from low-resolution images that are being video streamed \cite{huang2008three}.

Formally, at the $k$-th iteration, we assume that a set of rows of $\bfA$ and corresponding elements of $\bfb$ become available, which we denote by $\bfW_k\t\bfA$ and $\bfW_k\t\bfb$ respectively. Here the matrix $\bfW_k \in \bbR^{m \times \ell}$ can be seen as a \emph{sampling} matrix, which selects rows of $\bfA$ and $\bfb$.
For a fixed $M \in \mathbb{N}$ we assume that matrices $\{\bfW_{i}\}_{i=1}^{M}$  satisfy the following properties:
\begin{enumerate}
\item for each $i\in \{1,\ldots,M\}$, $\bfW_{i} \in \mathbb{R}^{m \times \ell}$, where $\ell = \frac{m}{M}$\footnote{To avoid a notational distraction, we assume all matrices $\bfW_i$ are of the same dimension and $\ell M = m$; hence, $\ell\in \bbN$. However a generalization with different matrix sizes $\bfW_i \in \bbR^{m\times\ell_i}$ is straightforward.} and
\item the sum $\sum_{i=1}^{M} \bfW_{i}\bfW\t_{i}= \bfI_{m}$.
\end{enumerate}
The first assumption implies that the size of $\bfW\t_{i}\bfA$
is smaller than the size of $\bfA$, and thus computationally manageable. The second assumption guarantees that all rows of $\bfA$ are given equal weight overall.

Notice that if $\bfW_k$ is sparse with only a few non-zero elements in a subset of the $m$ columns, $\bfW_k\t\bfA$ extracts only rows of $\bfA$ where $\bfW_k$ has nonzero entries. Hence, these methods are commonly known as row action methods \cite{escalante2011alternating,andersen2014generalized}. Randomized or sketching methods are also related in that a single realization of $\bfW_k$ is used to project a large system onto a small dimensional subspace \cite{drineas2012fast,pilanci2016iterative}.
However, these methods typically require access to all of the data at once (e.g. $\bfW_{k}$ is not sparse).

\paragraph{Overview and Outline.} In this paper, we describe iterative sampling methods for solving Tikhonov-regularized problems, where the main distinction from existing methods such as hybrid Krylov methods and iterated Tikhonov methods is that we do not require ``all-at-once'' access to the forward model.  In terms of theoretical results, the main contributions include the characterization of iterates as solutions to partial or full Tikhonov problems and asymptotic convergence results.  In terms of methodology, we highlight the sampled Tikhonov method where the regularization parameter can be updated during the iterative process and the iterates are Tikhonov-regularized solutions after each epoch of data. Additionally, the sampled Tikhonov method converges asymptotically to a Tikhonov-regularized solution.  Other developments include methods for updating the regularization parameter using sampled data and limited-memory variants for problems with many unknowns.

The paper is organized as follows.  In Section~\ref{sec:rrls} we describe two iterative methods for Tikhonov regularization with sampling.  Various theoretical results are provided, including asymptotic convergence results.  In Section~\ref{sec:regparselection} we describe sampled regularization parameter selection methods that can be used to update and adapt the regularization parameter.  Numerical illustrations are provided throughout, and a limited-memory variant of these methods is described in Section~\ref{sec:numerics}, along with results for a large-scale imaging problem. Conclusions and future work are discussed in Section~\ref{sec:conclusions}.

\section{Iterative sampling methods for Tikhonov regularization}
\label{sec:rrls}
Iterative sampling methods for Tikhonov regularization can be used to solve massive linear inverse problems.  We will investigate two methods.
Let $\bfy_0,\,\bfx_0\in \bbR^n$ be initial iterates and let $\bfW_i \in \bbR^{m \times \ell}$, $i =1,\ldots, k$ be arbitrary matrices. For notational convenience, we denote $\bfA_{i} = \bfW\t_{i}\bfA$ and $\bfb_{i} = \bfW\t_{i}\bfb$.  Assuming a fixed regularization parameter $\lambda$, the first method that we consider is \emph{regularized recursive least squares} (\texttt{rrls}), which is defined as
\begin{equation} \label{eq:sTik1}
  \bfy_k = \bfy_{k-1} - \bfB_k\bfA_{k}\t(\bfA_{k}\bfy_{k-1}-\bfb_{k}),\quad k \in \bbN,
\end{equation}
where $\bfB_k = \left( \lambda\bfL\t\bfL + \sum_{i=1}^{k}\bfA_{i}\t\bfA_{i}\right)^{-1}$.
If $\bfW_i$ is the $i$-th column of the identity matrix, \texttt{rrls} is an extension of the recursive least squares algorithm \cite{bjorck1996numerical} that includes a Tikhonov term. Since it may be difficult to know a good regularization parameter in advance, we propose a \emph{sampled Tikhonov} (\texttt{sTik}) method, where the iterates are defined as
\begin{equation} \label{eq:sTik2}
  \bfx_k = \bfx_{k-1} - \bfB_k\left(\bfA\t_{k}(\bfA_{k}\bfx_{k-1}-\bfb_{k})+\Lambda_k\bfL\t\bfL\bfx_{k-1}\right),\quad k \in \bbN,
\end{equation}
where $\bfB_k = \left( \sum_{i=1}^{k}\Lambda_i\bfL\t\bfL + \sum_{i=1}^{k}\bfA_{i}\t\bfA_{i}\right)^{-1}.$  Compared to \texttt{rrls}, the main advantages of the \texttt{sTik} method are that the regularization parameter can be updated during the iterative process and that in a sampled framework, the \texttt{sTik} iterates converge asymptotically to a Tikhonov solution whereas the \texttt{rrls} iterates converge asymptotically to an unregularized solution.
Of course, selecting a good regularization parameter can be difficult, especially for problems with a small range of good values.  However, if a good parameter estimate is available or can be estimated (see Section~\ref{sec:regparselection}), it is desirable that the numerical method for solution computation converges to a regularized solution.

In this section, we begin by showing that for arbitrary matrices $\bfW_i$, both \texttt{rrls} and \texttt{sTik} iterates can be recast as solutions to regularized least squares problems (c.f. Theorem~\ref{thm:sqnrls}). See \ref{sec:appendix1} for proofs for all theorems from Section~\ref{sec:rrls}.

\begin{theorem}
\label{thm:sqnrls}
  Let $\bfA \in \bbR^{m \times n}$ and $\bfb \in \bbR^{m}$.  Let $\bfL\in\bbR^{s\times n}$ have full column rank and $\bfW_i \in \bbR^{m \times \ell}$, $i =1,\ldots, k$ be an arbitrary sequence of matrices.
  \begin{itemize}
    \item[(i)] For $\lambda >0$ and $\bfy_0\in \bbR^n$ arbitrary, the \texttt{rrls} iterate~\eqref{eq:sTik1} with $\bfB_k = \left( \lambda\bfL\t\bfL + \sum_{i=1}^{k}\bfA_{i}\t\bfA_{i}\right)^{-1}$ is the solution of the least squares problem
  \begin{equation}\label{eq:lsk}
    \min_{\bfx} \ \ \norm[2]{[\bfW_1,\ldots, \bfW_k ]\t(\bfA\bfx -\bfb)}^2+ \lambda\norm[2]{\bfL(\bfx-\bfy_{0})}^{2}.
  \end{equation}
    \item[(ii)] For $\lambda_k = \sum_{i = 1}^{k}\Lambda_i >0$ for any $k$ and $\bfx_0\in \bbR^n$ arbitrary, the \texttt{sTik} iterate~\eqref{eq:sTik2} with $\bfB_k = \left( \sum_{i=1}^{k}\Lambda_i\bfL\t\bfL + \sum_{i=1}^{k}\bfA_{i}\t\bfA_{i}\right)^{-1}$ is the solution of the least squares problem
    \begin{equation}\label{eq:lskl}
        \min_{\bfx} \ \ \norm[2]{[\bfW_1,\ldots, \bfW_k ]\t(\bfA\bfx -\bfb)}^2+ \lambda_k\norm[2]{\bfL\bfx}^{2}.
    \end{equation}
  \end{itemize}
\end{theorem}

The above results are true for any arbitrary sequence of matrices $\{\bfW_k\}$.  Next, we consider a fixed set of matrices, as described in the introduction, and allow random sampling from this set.  To be precise, define $\bfW_{\tau(k)}$ to be a random variable at the $k$-th iteration, where
 $\tau(k)$ is a random variable that indicates a sampling strategy.  For example, if we let $\tau(k)$ be a uniform random variable on the set $\{1, \dots, M\}$, then we would be sampling with replacement. For this random sampling strategy, we prove asymptotic convergence of \texttt{rrls} and \texttt{sTik} iterates in Section~\ref{subsec:withreplacement}.  Then, we focus on random cyclic sampling, where for each $j \in \mathbb{N}$, $\{\tau(k)\}_{jM+1}^{(j+1)M}$ is a random permutation on the set $\{1, \dots, M\}$. Note, cyclic sampling, where $\tau(k) = {k \mod M}$, is a special case of random cyclic sampling.
 We note that, until all blocks have been sampled, random cyclic sampling is nothing more than sampling without replacement.  For random cyclic sampling, we characterize iterates after each epoch and prove asymptotic convergence of \texttt{rrls} and \texttt{sTik} iterates in Section~\ref{subsec:withoutreplacement}.  An illustrative example comparing the behavior of the solutions is provided in Section~\ref{subsec:illustration}. For notational simplicity we denote $\bfA_{\tau(k)} = \bfW\t_{\tau(k)}\bfA$ and $\bfb_{\tau(k)} = \bfW\t_{\tau(k)}\bfb$.

Notice that for both random sampling and random cyclic sampling, we have the following property,
 \begin{equation}\label{eq:Wexpectation}
   \bbE  \, \bfW_{\tau(k)}\bfW_{\tau(k)}\t = \frac{1}{M}\bfI_{m} = \frac{\ell}{m}\bfI_m\,.
 \end{equation}
There are many choices for $\{\bfW_{i}\}$, see e.g., \cite{chung2017stochastic, le2017data, matouvsek2008variants}, but a simple choice is a block column partition of a permutation matrix. For this choice of $\{\bfW_{i}\}$, $\bfA_{\tau(k)}$ is just a predefined block of rows of $\bfA$.
This is the primary choice of  $\{\bfW_{i}\}$ we will consider.

\subsection{Random sampling}
\label{subsec:withreplacement}
Next we investigate the asymptotic convergence of \texttt{rrls} and \texttt{sTik} iterates for the case of random sampling.  This is also referred to as sampling with replacement.

\begin{theorem}\label{thm:converge_xls}
  Let $\bfA\in\bbR^{m\times n}$ and $\bfb\in\bbR^m$.  Let $\bfL\in\bbR^{s\times n}$ have full column rank and $\{ \bfW_{i}\}_{i=1}^{M}$ be a set of real valued $m \times \ell$ matrices with the property that $\sum_{i=1}^{M} \bfW_{i} \bfW_{i}\t = \bfI_m$, and let $\tau(k)$
  be a uniform random variable on the set $\{1, \dots M\}$.
   \begin{itemize}
     \item[(i)] Let $\lambda >0$, $\bfy_0\in \bbR^n$ be arbitrary, and define the sequence $\{\bfy_k\}$ as
    \begin{equation}
      \label{eq:rrls_tau}
      \bfy_k = \bfy_{k-1} - \bfB_k\bfA_{\tau(k)}\t(\bfA_{\tau(k)}\bfy_{k-1}-\bfb_{\tau(k)}),\quad k \in \bbN,
    \end{equation}
    where $\bfB_k = \left( \lambda\bfL\t\bfL + \sum_{i=1}^{k}\bfA_{\tau(i)}\t\bfA_{\tau(i)}\right)^{-1}$.  If $\bfA$ has full column rank, then $\bfy_k \arras\bfx(0).$
  \item [(ii)] Let $\sum_{i=1}^k \Lambda_i >0$ for all $k$, and $\lambda = \lim_{k\to\infty} \tfrac{M}{k} \sum_{i=1}^k \Lambda_i >0$ be finite. Let $\bfx_0\in \bbR^n$ be arbitrary, and define the sequence $\{\bfx_k\}$ as
  \begin{equation}
    \label{eq:sTik_tau}
    \bfx_k = \bfx_{k-1} - \bfB_k\left(\bfA_{\tau(k)}\t(\bfA_{\tau(k)}\bfx_{k-1}-\bfb_{\tau(k)})+\Lambda_k\bfL\t\bfL\bfx_{k-1}\right),
  \end{equation}
  where $\bfB_k = \left( \sum_{i=1}^{k} \Lambda_i\bfL\t\bfL+ \sum_{i=1}^{k} \bfA_{\tau(i)}\t\bfA_{\tau(i)}\right)^{-1}$.  Then
$\bfx_k \arras\bfx(\lambda)$.
   \end{itemize}
\end{theorem}
The significance of Theorem~\ref{thm:converge_xls} is that the \texttt{rrls} iterates converge asymptotically to the \emph{unregularized} least-squares solution, $(\bfA\t \bfA)^{-1}\bfA\t \bfb$, which is undesirable for ill-posed inverse problems.
On the other hand, the \texttt{sTik} iterates converge asymptotically to a Tikhonov-regularized solution. Note that for a given $\lambda$, convergence to $\bfx\left(\lambda\right)$ is ensured by  setting  $\Lambda_k=\frac{\lambda}{M}$. A more realistic scenario would be to adapt $\Lambda_{k}$ as data become available, since the desired regularization parameter is typically not known before the data is received. Parameter selection strategies for selecting $\Lambda_k$ are addressed in Section~\ref{sec:regparselection}, and empirically we observe favorable asymptotic properties (see Example 2).

\subsection{Random Cyclic Sampling}
\label{subsec:withoutreplacement}
Next we investigate \texttt{rrls} and \texttt{sTik} with random cyclic sampling.
In addition to proving asymptotic convergence in this case, we can also describe the iterates as Tikhonov solutions after each epoch, where an epoch is defined as a sweep through all the data.

\begin{theorem} \label{thm:scwr}
  Let $\bfA\in\bbR^{m\times n}$ and $\bfb\in\bbR^m$.  Let $\bfL\in\bbR^{s\times n}$ have full column rank and $\{ \bfW_{i}\}_{i=1}^{M}$ be a set of real valued $m \times \ell$ matrices with the property that $\sum_{i=1}^{M} \bfW_{i} \bfW_{i}\t = \bfI_m$, and let $\tau(k)$ be a random variable such that for $j \in \mathbb{N}$, $\{\tau(k)\}_{jM+1}^{(j+1)M}$ is a random permutation on the set $\{1, \dots, M\}.$
\begin{enumerate}
  \item If $\lambda >0$, $\bfy_0 = \bfzero$, and the sequence $\{\bfy_k\}$ is defined as~\eqref{eq:rrls_tau} with $\bfB_k = \left( \lambda\bfL\t\bfL + \sum_{i=1}^{k}\bfA_{\tau(i)}\t\bfA_{\tau(i)}\right)^{-1}$, then the \textnormal{\texttt{rrls}} iterate at the $j$-th epoch is given as \ $\bfy_{jM} =  \bfx\left(\tfrac{1}{j}\lambda\right)$.
   \item Let $\{\Lambda_k\}$ be an infinite sequence with the property that
 $\lambda_k = \sum_{i=1}^k \Lambda_i >0$. If $\bfx_0$ is arbitrary and the sequence $\{\bfx_k\}$ is defined as~\eqref{eq:sTik_tau} with $\bfB_k = \left( \sum_{i=1}^{k} \Lambda_{i}\bfL\t\bfL + \sum_{i=1}^{k}\bfA_{\tau(i)}\t\bfA_{\tau(i)}\right)^{-1}$, then the \textnormal{\texttt{sTik}} iterate at the $j$-th epoch is given as $\bfx_{jM} =  \bfx\left(\tfrac{1}{j}\lambda_{jM}\right)$.
 \end{enumerate}
\end{theorem}

Notice that at every epoch, the effective regularization parameter for \rrls, i.e., $\tfrac{\lambda}{j}$, is reduced.  Also, if $\bfA$ has full column rank, we have $\lim_{j\to \infty}\bfy_{jM} =\bfx(0).$
On the other hand, the \sTik iterates do not converge to the unregularized solution but do converge to a Tikhonov-regularized solution, since at each epoch $j = k/M$ and we have $\bfx_{jM} = \bfx_{k} =  \bfx\left(\tfrac{M}{k}\lambda_{k}\right)$ and $\tfrac{M}{k}\lambda_{k}>0$. In Section~\ref{subsec:illustration}, we illustrate the convergence behavior of the \rrls and \sTik iterates, but first we make some connections to existing optimization methods.

\subsection{Connections to stochastic approximation methods}
\label{sec:SO}
There is a connection between the iterative methods with sampling presented in Section \ref{sec:rrls} and stochastic approximation methods. First we recast the Tikhonov problem~\eqref{eq:ip} as a stochastic optimization problem.  For simplicity, consider random sampling (i.e., with replacement), where $\tau(k)$ is a uniform random variable on the set $\{1, \dots , M\}$.
Then if we define $ f_{\tau(k)}(\bfx) = \norm[2]{\bfW_{\tau(k)}\t\left(\bfA\bfx - \bfb\right)}^2 + \frac{\lambda}{M}\norm[2]{\bfL\bfx}^{2}$, it is easy to show that

\begin{equation*}
 \bbE \, f_{\tau(k)} \propto f\,,
  \end{equation*}
and therefore
\begin{equation} \label{eq:so}
  \argmin_{\bfx} \ \bbE \, f_{\tau(k)}(\bfx)  = \argmin_{\bfx} \  f(\bfx).
\end{equation}
There are a number of stochastic optimization methods that can be used to compute solutions to the expectation minimization problem on the  left. Stochastic approximation methods represent one class of methods \cite{shapiro2009lectures}. For the Tikhonov problem, a stochastic approximation method has the form,
\begin{equation}
  \label{eq:sa}\bfx_{k+1}  = \bfx_{k} + \bfB_{k}\nabla f_{\tau(k)}\left(\bfx_{k}\right)\,,
\end{equation}
where $\nabla f_{\tau(k)}\left(\bfx_{k}\right) = \bfA\t_{\tau(k)}\left(\bfA_{\tau(k)}\bfx_{k}-\bfb_{\tau(k)}\right) + \frac{\lambda}{M}\bfL\bfx_{k}$ is the sample gradient for the Tikhonov problem.
Different choices of $\bfB_k$ can be used in~\eqref{eq:sa}.  If $\bfB_{k} =  \left(\frac{k\lambda}{M}\bfL\t\bfL + \sum_{i=1}^{k}\bfA_{\tau(i)}\t\bfA_{\tau(i)}\right)^{-1}$, then all of the previously computed global curvature information is encoded in $\bfB_k$ and we recover the \texttt{sTik} method with $\Lambda_{i} = \frac{\lambda}{M}$.  Theorem \ref{thm:converge_xls} (ii) shows that these iterates will converge asymptotically to the minimizer of \eqref{eq:so}, but storage can get costly.
Another option is to take $\bfB_{k} = \bfI_n$, which corresponds to the stochastic gradient method \cite{Bottou1998}.
For faster convergence closer to the minimizer,
there are various methods in the stochastic optimization literature that can be used to approximate the global curvature information $\nabla f \left(\bfx_{k}\right)$ \cite{bottou2004large, Kushner1997}.  For example, a stochastic LBFGS method stores a small set of vectors, rather than matrix $\bfB_{k}$, and can perform multiplications in an efficient manner \cite{mokhtari2015global, Byrd2016}.

We are most interested in the Tikhonov problem~\eqref{eq:ip}, but we note that there exists methods for the case where $\lambda=0$  that have connections to stochastic optimization methods.  Using the same reformulation as above, a stochastic approximation method would have the form~\eqref{eq:sa}.
If we take $\bfB_{k} = \left(\nabla^{2} f_{\tau(k)}\right)^{\dagger}$, then we get the randomized block Kaczmarz method \cite{Needell2014, strohmer2009randomized,andersen2014generalized}.  Notice that the curvature information comes only from the current sample.
On the other hand, if $\bfB_{k}$ is chosen to contain all previous curvature information, we get the \texttt{rrls} iterates
 \begin{equation*}
  \bfy_k = \bfy_{k-1} - \left( \lambda\bfL\t\bfL + \sum_{i=1}^{k}\bfA_{\tau(i)}\t\bfA_{\tau(i)}\right)^{-1} \bfA_{\tau(k)}\t(\bfA_{\tau(k)}\bfy_{k-1}-\bfb_{\tau(k)}).
\end{equation*}

Note that $\lambda\bfL\t\bfL$ is included to ensure invertibility and is often replaced with $\bfI_n$.  Regardless, the iterates converge to the unregularized problem, c.f., Theorem \ref{thm:converge_xls} (i).
The connection between recursive least squares and stochastic approximation methods was noted in \cite{Kushner1997}, and the approximation can be interpreted as a regularized stochastic approximation method that was considered, e.g., in \cite{chung2017stochastic,bottou2018optimization}.

\subsection{An Illustration} \label{subsec:illustration}
In the following illustration, we use a small toy example to highlight the convergence behaviors of \rrls and \sTik iterates.  We investigate both random sampling and random cyclic sampling, and we demonstrate convergence by plotting solutions after multiple epochs of the data.
The example we use is a Tikhonov problem of the form~\eqref{eq:ip}, where
\begin{equation*}
  \bfA = \begin{bmatrix}\bfone  &  \bfdelta_\bfA \\ 0 & 1 \end{bmatrix} \in \bbR^{10\times 2}, \qquad \bfb = \bfA\bfx_{\rm true}+\bfdelta_\bfb, \qquad \mbox{and }\quad \bfx_{\rm true} = \bfone.
\end{equation*}
The vectors $\bfdelta_\bfA$ and $\bfdelta_\bfb$ are realizations from the normal distributions $\calN(\bfzero, 0.005\,\bfI_{9})$ and $\calN(\bfzero, 0.1\,\bfI_{10})$ respectively, and $\bfone$ is the vector of ones of appropriate length. We further choose $\bfL = \bfI_2$ and fix $\lambda = 0.2$ for the \rrls iterates $\bfy_k$. For \sTik iterates $\bfx_k$, we choose the parameters $\Lambda_k$ such that the regularization is constant at each epoch, i.e.,  $\frac{10}{k} \sum_{i= 1}^k \Lambda_i = 0.2$. With this setup we have $\bfx(0) = [1.0869, -1.3799]\t$ and $\bfx(\lambda) = [1.0698, -0.0271]\t$.
We let $\bfW_{\tau(i)}$ be the $\tau(i)$-th column of the identity matrix, and set $\bfx_0 = \bfy_0 = \bfzero$.

\begin{figure}[b]
  \begin{center}
   \includegraphics[width=\textwidth]{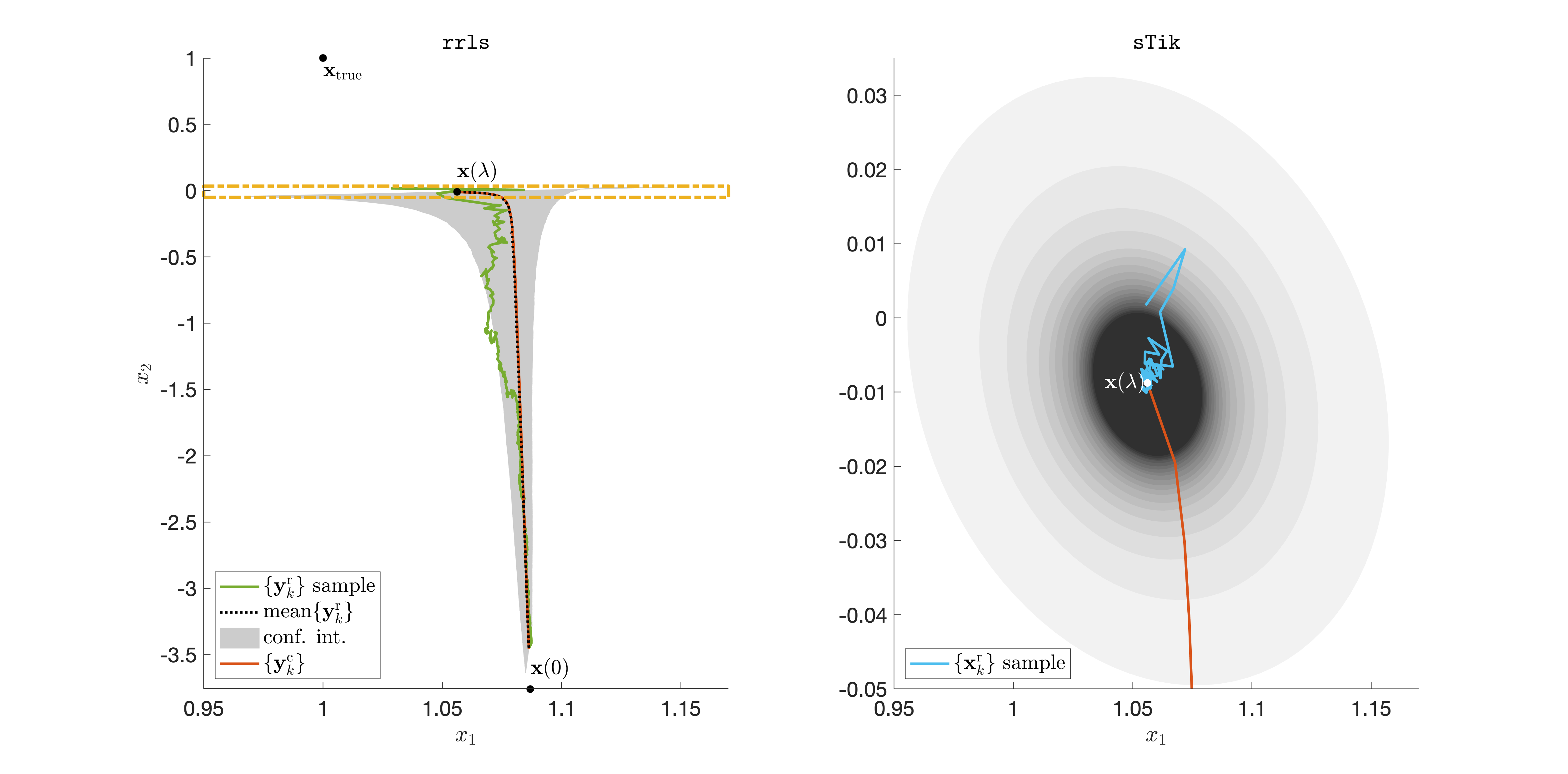}
  \end{center}
\caption{Illustration of convergence behaviors of \rrls and \sTik iterates.  Shown in the left panel are the true solution $\bfx_{\rm true}$, the unregularized solution $\bfx(0)$, the Tikhonov solution $\bfx(\lambda)$, and \rrls iterates after multiple epochs.
Both \rrls with random sampling iterates $\{\bfy_k^{\rm r}\}$ and \rrls with random cyclic sampling iterates $\{\bfy_k^{\rm c}\}$ converge asymptotically to the unregularized solution.
In the right panel, we provide \sTik with random sampling iterates $\{\bfx_k^{\rm r}\}$ and confidence bounds.  These iterates stay close to the Tikhonov solution.  The axis for the right figure corresponds to the rectangular box in the left figure. The concentric gray circles represent the $90\%$ confidence interval for these iterates after subsequent epochs.  }\label{fig:toyexample}
\end{figure}

In Figure~\ref{fig:toyexample}, we provide two illustrations.  In the left panel, we provide the true solution $\bfx_{\rm true}$, the unregularized solution $\bfx(0)$, the Tikhonov solution $\bfx(\lambda)$, and the \rrls iterates after each epoch.  The \rrls iterates with random sampling with replacement are denoted by $\bfy_k^{\rm r}$, and the \rrls iterates with random cyclic sampling are denoted by $\bfy_k^{\rm c}$. Notice that by Theorem~\ref{thm:scwr}, $\bfy_k^{\rm c}$ at each epoch is a Tikhonov solution, i.e., after the $j$-th epoch $\bfy_{jM}^{\rm c} = \bfx\left(\tfrac{1}{j}\lambda\right)$.  Thus, we get a set of Tikhonov solutions with vanishing regularization parameters, and these iterates asymptotically converge to the unregularized solution.
For \rrls with random sampling, we run $1,\!000$ simulations and provide one sample path, along with the mean (dotted line) and region of the $95$-th percentile shaded in grey. We note that the mean of $\{\bfy_k^{\rm r}\}$ is almost identical to the random cyclic sequence $\{\bfy_{k}^{\rm c}\}$ (orange line) suggesting that the random sequence $\{\bfy_{k}^{\rm r}\}$ is an unbiased estimator of the deterministic sequence $\{\bfy_{k}^{\rm c}\}$ (at each epoch). In the right panel of Figure~\ref{fig:toyexample}, we provide the \sTik iterates with random sampling, which are denoted by $\bfx_k^{\rm r}$.  Again, we run $1,\!000$ simulations and provide one simulation along with the shaded percentiles.  It is evident that with more epochs, the iterates get closer to the desired Tikhonov solution. To aid with visual scaling, the axis for the right figure corresponds to the dotted rectangular box in the left figure.  The \sTik iterates with random cyclic sampling are omitted since $\bfx_{jM}^{\rm c} = \bfx(\lambda)$ (i.e., we get the Tikhonov solution after each epoch).

 We observe that for random sampling, both \rrls and \sTik iterates contain undesirable uncertainties in the estimates. Although \rrls iterates provide approximations to the Tikhonov solution, the main disadvantages are that the regularization parameter cannot be updated during the process and the iterates converge asymptotically to the unregularized solution. Hence we disregard the \rrls method and focus on \sTik with \emph{random cyclic sampling}, where  $\lambda$ can be updated through $\Lambda_k$.

 \section{Sampled regularization parameter selection methods}
 \label{sec:regparselection}
 The ability to update the regularization parameter while still exhibiting favorable  convergence properties makes the \sTik method appealing for massive inverse problems. However, sampled regularization parameter selection methods must be developed to enable proper choices of updates $\Lambda_k$. Unfortunately, standard regularization parameter selection methods are not feasible in this setting because many of them require access to the full residual vector, $\bfr(\lambda) = \bfA\bfx(\lambda) -\bfb$, which is not available.  In this section, we investigate variants of existing regularization parameter selection methods \cite{aster2018parameter, tenorio2017introduction,Bardsley2018} that are based on the sample residual.

 In the following we assume that at the $k$-th iteration $\Lambda_i$, $i= 1,\ldots,k-1$ have been determined.
 Then the goal is to determine an appropriate update parameter $\Lambda_k$. Notice that from Theorems~\ref{thm:sqnrls} and~\ref{thm:scwr}, the $k$-th \sTik iterate can be represented as
 \begin{gather}
   \label{eqn:regularized}
   \bfx_k(\lambda) = \bfC_k(\lambda) \bfb, \quad \mbox{where}\\
   \bfC_k(\lambda) = \left(\left(\lambda + \sum_{i = 1}^{k-1}\Lambda_i \right)\bfL\t\bfL + \sum_{i=1}^k \bfA\t \bfW_{\tau(i)} \bfW_{\tau(i)}\t \bfA\right)^{-1} \sum_{i=1}^k \bfA\t \bfW_{\tau(i)} \bfW_{\tau(i)}\t\,.  \nonumber
 \end{gather}

 \paragraph{\bf Sampled discrepancy principle.} The basic idea of the \emph{sampled discrepancy principle (sDP)} is that at the $k$-th iteration, the goal is to select parameter $\Lambda_k$ so that the sum of squared residuals for the current sample $\norm[2]{\bfW_{\tau(k)}\t(\bfA\bfx_k - \bfb)}^2$ is equal to $\mathbb{E}\norm[2]{\bfW_{\tau(k)}\t \bfepsilon}^2$. Using properties of conditional expectation, we find
 \begin{align*} \mathbb{E} \norm[2]{\bfW_{\tau(k)}\t\left(\bfA\bfx_{\text{true}}-\bfb\right)}^{2}  = &\, \mathbb{E}\norm[2]{\bfW_{\tau(k)}\t\bfepsilon}^{2}  \\  = & \, \mathbb{E} \, \mathbb{E}\left[\bfepsilon\t\bfW_{\tau(k)}\bfW_{\tau(k)}\t\bfepsilon \, \rvert \, \bfepsilon \right]  \\
   = & \, \sigma^2 \trace{\mathbb{E} \, \bfW_{\tau(k)}\bfW_{\tau(k)}\t } \\
   = & \, \sigma^{2}\ell,\end{align*}
 where $\trace{\cdot}$ corresponds to the matrix trace function.
 Thus, at the $k$-th iteration and for a given realization, we select $\lambda$ such that
 \begin{align*}
 \norm[2]{\bfW_{\tau(k)}\t\left(\bfA\bfx_k(\lambda)-\bfb\right)}^{2} \approx \gamma \sigma^{2} \ell\,,
 \end{align*}
 where $\gamma>1$ is some predetermined real number, see \cite{hansen2006deblurring,tenorio2017introduction} for details.

 \paragraph{\bf Sampled unbiased predictive risk estimator.} Next, we describe a method to select $\Lambda_k$ based on a \emph{sampled unbiased predictive risk estimator (sUPRE)}.
 The basic idea is to find $\Lambda_k$ to minimize the sampled predictive risk,
 \begin{equation*}
   \bbE  \norm[2]{\bfW_{\tau(k)}\t(\bfA \bfx_k(\lambda) -\bfA\bfx_{\true})}^2 \,,
 \end{equation*}
 which is equivalent to
 \begin{equation*}
   \bbE \norm[2]{\bfW_{\tau(k)}\t\left(\bfA \bfx_k(\lambda) -\bfb\right)}^2 + 2 \sigma^2\, \bbE\, \trace{\bfW_{\tau(k)} \bfW_{\tau(k)}\t \bfA\bfC_k(\lambda)} - \sigma^2 \ell\,.
 \end{equation*}
 See~\ref{sec:SUPRE} for details of the derivation.  Then, similar to the approach used in the standard UPRE derivation, the parameter $\Lambda_k$ is selected by finding a minimizer of the unbiased estimator for the sampled predictive risk,
 \begin{equation}
   \label{eqn:sampledUPRE}
   U_k(\lambda) = \norm[2]{\bfW_{\tau(k)}\t\left(\bfA \bfx_k(\lambda) -\bfb\right)}^2 + 2 \sigma^2 \trace{\bfW_{\tau(k)}\t \bfA\bfC_k(\lambda)\bfW_{\tau(k)}} - \sigma^2\ell \,,
 \end{equation}
 for a given realization. Similar to the standard methods, both sDP and sUPRE require estimates of $\sigma^2$.  For Gaussian noise, there are various ways that one can obtain such an estimate, see e.g.,~\cite{donoho1995noising, tenorio2017introduction}.

 \paragraph{\bf Sampled generalized cross validation.} Lastly, we describe the \emph{sampled generalized cross validation (sGCV)} method for selecting $\Lambda_k$ and point the interested reader to~\ref{sec:SGCV} for details of the derivation.  The basic idea is to use a ``leave-one-out'' cross validation approach to find a value of $\Lambda_k$, but the main differences compared to the standard GCV method are that at the $k$-th iteration, we only have access to the sample residual and the iterates only correspond to Tikhonov solutions with only partial data.  The parameter $\lambda_k$ is selected by finding a minimizer of the sGCV function,
 \begin{equation}
   \label{eq:sampledGCV}
   G_k(\lambda)=\frac{\ell \norm[2]{\bfW_{\tau(k)}\t(\bfA\bfx_k(\lambda) - \bfb)}^2}{\trace{\bfI_\ell-\bfW_{\tau(k)}\t\bfA\bfC_k(\lambda)\bfW_{\tau(k)}}^2}=\frac{\ell \norm[2]{\bfW_{\tau(k)}\t(\bfA\bfx_k(\lambda) - \bfb)}^2}{\left(\ell-\trace{\bfW_{\tau(k)}\t\bfA\bfC_k(\lambda)\bfW_{\tau(k)}}\right)^2}\,.
 \end{equation}

 Adapting regularization parameters during an iterative processes is not a new concept; however, much of the previous work in this area utilize projected systems, see e.g., \cite{renaut2017hybrid,kilmer2001choosing}, or are specialized to applications such as denoising \cite{hashemi2015adaptive}.  Another common approach is to consider the unregularized problem and to terminate the iterative process before noise contaminates the solution.  This phenomenon is called semiconvergence, and selecting a good stopping iteration can be very difficult.  There have been investigations into semiconvergence behavior of iterative methods such as Kaczmarz, e.g., \cite{elfving2014semi}.

 \paragraph{Example 2.} In this example, we investigate the behavior of the previously discussed sampled regularization parameter update strategies, i.e., sDP, sUPRE, and sGCV, for  multiple ill-posed inverse problems from the Matlab matrix gallery and from P.~C.~Hansens' Regularization Tools toolbox \cite{regtools,matlabgallery}. For simplicity, we set $m = n = 100$ and use the true solutions $\bfx_{\rm true}$ that are provided by the toolbox.  If none is provided, we set $\bfx_{\rm true} = \bfone$. We let $\bfL = \bfI_{100}$, and set $\bfepsilon \sim \calN(\bfzero, 0.01\,\bfI_{100})$.
 Sampling matrices $\bfW_j \in \bbR^{100 \times 10}$ are given as $\bfW_j = [\bfzero_{ 10(j-1)\times 10}; \bfI_{10}; \bfzero_{10(10-j)\times 10}]$ for $j = 1,\ldots, 10$, such that $\bfA$ and $\bfb$ are sampled in 10 consecutive blocks.  Here, we sample $\bfW$ in a random cyclic fashion and let $\sigma^2$ be the true noise variance for sDP and sUPRE.  For sDP, we set $\gamma = 4$, in accordance with \cite{hansen2006deblurring,tenorio2017introduction}.

 We first consider the {\tt prolate} example, where $\bfA$ is an ill-conditioned Toeplitz matrix that comes from Matlab's matrix gallery.
 In Figure~\ref{fig:example2a} we illustrate the asymptotic behavior of these sampled parameter selection strategies by plotting the number of epochs versus the value of $\lambda$ for sDP, sUPRE, and sGCV.
 For comparison, we provide the regularization parameter for the full problem corresponding to DP, UPRE, and GCV.  DP and UPRE use the true noise variance, and $\gamma$ is as above for DP.  For comparison, we also provide the optimal parameter $\lambda_{\rm opt}$ for the full problem, which is the parameter that minimizes the 2-norm of the error between the reconstruction and the true solution.  This last approach is not possible in practice. Empirically, we observe that with more iterations, the sampled regularization parameter selection methods tend to ``stabilize''. The sDP regularization parameter stabilizes near the DP parameter for the full problem, but both sUPRE and sGCV stabilize closer to the optimal regularization parameter.

 \begin{figure}[b]
   \begin{center}
    \includegraphics[width=0.9\textwidth]{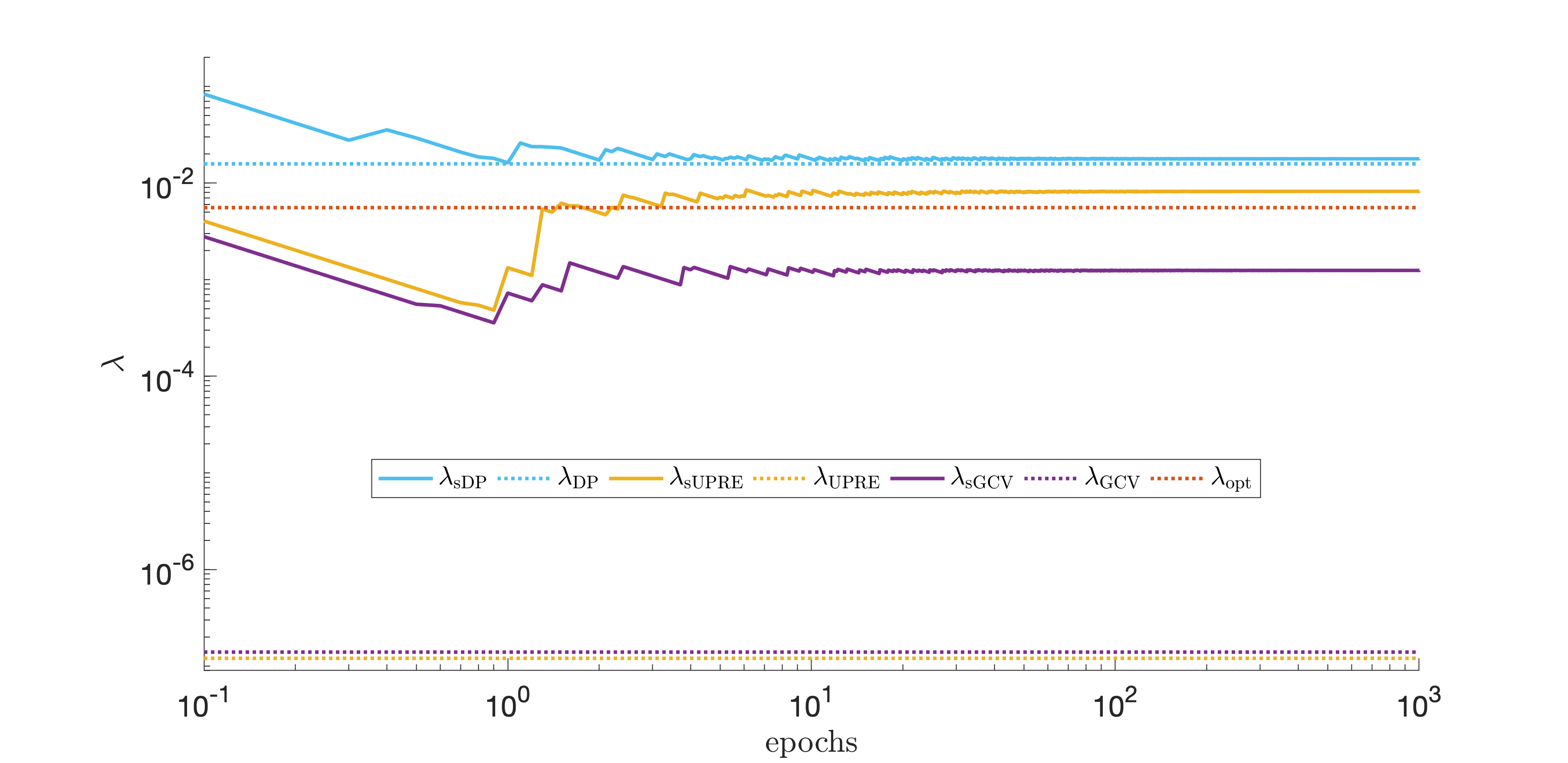}
   \end{center}
   \caption{``Asymptotic'' behavior of the sampled regularization parameter selection methods for the {\tt prolate} example.  Corresponding regularization parameters computed using the full data are provided as horizontal lines for comparison.
     }\label{fig:example2a}
 \end{figure}

 \begin{figure}[b]
   \begin{center}
    \includegraphics[width=1.1\textwidth]{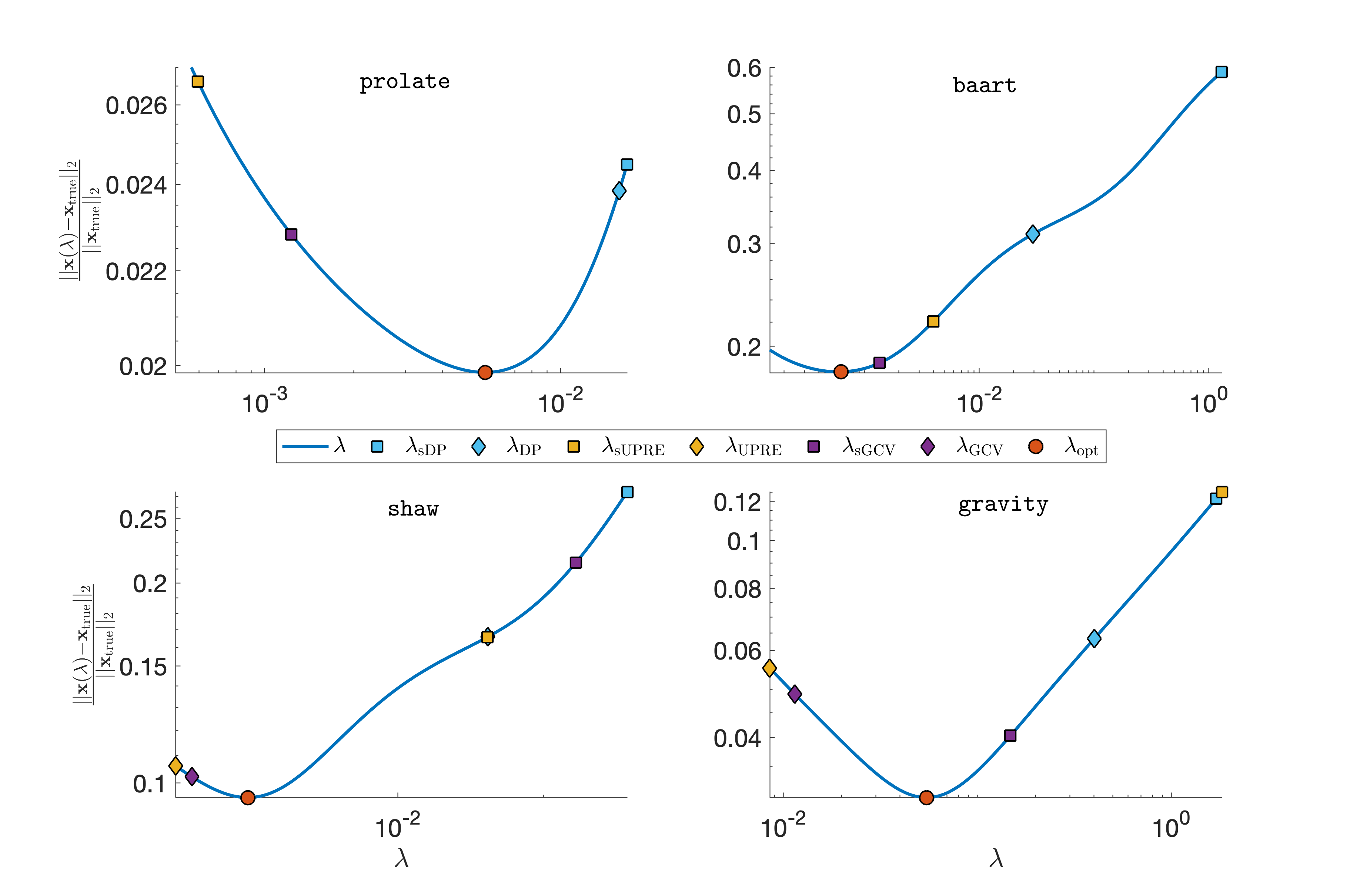}
   \end{center}
   \caption{Relative reconstruction errors of the sampled and full regularization methods for four test problems {\tt prolate}, {\tt baart}, {\tt shaw}, and {\tt gravity}. All solutions lie on the solid line, which corresponds to relative errors for Tikhonov solutions. Note that the UPRE and GCV estimation in the {\tt prolate} and {\tt baart} test problem underperform significantly and is therefore omitted. The relative errors for $\lambda_{\rm sUPRE}$ and $\lambda_{\rm DP}$ coincide in the {\tt shaw} example.}\label{fig:example2b}
 \end{figure}

 While we observe similar results for other test problems (results not shown), the sampled regularization parameters may not necessarily be close to the corresponding parameter for the full system.
 Nevertheless, the sampled regularization parameter selection methods often lead to appropriate reconstructions $\bfx_k(\lambda)$ after a moderate number of iterations $k$.
 Next, we investigate the relative reconstruction error $\norm[2]{\bfx_{10}(\lambda) -\bfx_{\rm true}}/\norm[2]{\bfx_{\rm true}}$ of sampled regularization methods after \emph{one} epoch. Figure~\ref{fig:example2b} illustrates results from four test problems ({\tt prolate}, {\tt baart}, {\tt shaw}, and {\tt gravity}). First note that by Theorem~\ref{thm:scwr}, all solutions are Tikhonov solutions for a $\lambda$ determined by the method, hence all relative reconstruction errors lie on a curve of relative errors for Tikhonov solutions.
 We observe that in terms of relative reconstruction error, the sampled regularization methods do not differ significantly from the full regularization methods. Note that all of the above parameter selection methods (including the standard DP, UPRE, and GCV) provide empirical estimations and thus may fail to provide good regularization parameters. Nevertheless, all of our sampled regularization parameter selection methods perform reasonably well on the test problems.

 \section{Numerical Results}
 \label{sec:numerics}
 For problems with a very large number of unknowns, such as those arising in pixel or voxel based image reconstruction, the described \sTik method is not feasible due to the construction of $n \times n$ matrix $\bfB_k.$  Although reduced models or subspace projection methods may be used to reduce the number of unknowns, obtaining a realistic basis for the solution may be difficult.  Instead, we describe various approximations of $\bfB_k$, which are related to methods described in Section~\ref{sec:SO}. In addition to being computationally feasible, all of these methods can take advantage of the adaptive regularization parameter selection methods described in Section~\ref{sec:regparselection}.

 These methods are based on the \sTik method.  In particular, we consider a sampled gradient (\texttt{sg}) method where the iterates are defined as~\eqref{eq:sTik2} where $\bfB_k = \left( \sum_{i=1}^{k}\Lambda_i\bfL\t\bfL + \bfI_n\right)^{-1}$ and a sampled block Kaczmarz (\texttt{sbK}) method where the iterates are defined as~\eqref{eq:sTik2} with $\bfB_k = \left( \sum_{i=1}^{k}\Lambda_i\bfL\t\bfL + \bfA_{k}\t\bfA_{k}\right)^{-1}$, hence just including the recent block $\bfA_k$. We also consider a \emph{limited-memory} version of \sTik called \texttt{slimTik}, which we describe below.  First notice that the $k$-th \sTik iterate is given by $\bfx_{k} = \bfx_{k-1} - \bfs_{k}$
 where
 \begin{align*}
 \bfs_{k} = \argmin_\bfs \norm[2]{ \begin{bmatrix}
   \bfA_{1} \\ \vdots \\ \bfA_{k-1} \\ \bfA_{k} \\ \sqrt{\sum_{i=1}^{k}\Lambda_i}\bfL
 \end{bmatrix}\bfs - \begin{bmatrix}
   \bf0\\ \vdots\\ \bf0 \\  \bfA_{k}\bfx_{k-1}-\bfb_{k} \\ \frac{\Lambda_k}{\sqrt{\sum_{i=1}^{k}\Lambda_i}} \bfL\bfx_{k-1}
 \end{bmatrix}  }^2.
 \end{align*}
 Note that with this reformulation, we must solve a least-squares problem with matrix
 $\begin{bmatrix}
   \bfA_{1}\t & \cdots & \bfA_{k}\t
 \end{bmatrix}\t$, which may get large for many samples.
 Thus, we select
 a memory parameter $r \in \mathbb{N}_0$ and define $\bfM_{k} = \begin{bmatrix}
   \bfA_{k-r}\t & \cdots & \bfA_{k-1}\t
 \end{bmatrix}\t \in \mathbb{R}^{r\ell \times n}$
 and $\bfA_{k-r} = \bf0$ for non-positive integers $k-r$. Then \slimTik iterates are given as $\bfx_{k} = \bfx_{k-1} - \tilde\bfs_{k}$ where
 \begin{align} \label{eq:slimTik}
 \tilde\bfs_{k} = \argmin_\bfs \norm[2]{ \begin{bmatrix}
   \bfM_k \\ \bfA_k\\ \sqrt{\sum_{i=1}^{k}\Lambda_i}\bfL
 \end{bmatrix}\bfs - \begin{bmatrix}
   \bfzero \\  \bfA_{k}\bfx_{k-1}-\bfb_{k} \\ \frac{\Lambda_k}{\sqrt{\sum_{i=1}^{k}\Lambda_i}} \bfL\bfx_{k-1}
 \end{bmatrix}  }^2.
 \end{align}
 Notice that in the case where $r=0$, \slimTik and \texttt{sbK} iterates are identical.
 First, we investigate the performance of {\tt sg}, {\tt sbK}, and {\tt slimTik} while taking advantage of the regularization parameter update described in Section~\ref{sec:regparselection}.  We use the {\tt gravity} example from Regularization Tools, where $\bfA \in \bbR^{1,000\times 1,000}$, $\bfL = \bfI_{1,000}$, and the noise level defined as $\frac{\norm[2]{\bfepsilon}}{\norm[2]{\bfA \bfx_{\rm true}}}$ is $0.01$.  The samples consist of $10$ blocks, each comprised of 100 consecutive rows of $\bfA$. The initial guess for the regularization parameter is chosen to be $0.1$ (the optimal overall regularization parameter in this example is approximately 0.0196), and we iterate for one epoch.

 In Figure~\ref{fig:example3} we provide the relative reconstruction errors per iteration for \texttt{sg}, \texttt{sbK}, \texttt{slimTik}, and \texttt{sTik}.
 Overall, we notice a correspondence between the amount of curvature information used to approximate the Hessian and an improvement in the relative reconstruction error.
 Although including more curvature results in greater computational costs and storage requirements, e.g., \texttt{sTik} may be infeasible for very large problems, the number of row accesses is the same for each method.
  In terms of parameter selection methods, sGCV performs better than sUPRE and sDP for this example.  The relative reconstruction error corresponding to the best overall Tikhonov solution is provided as the horizontal line.
 Although the results are not shown here, we note that the relative reconstruction errors will become very large for all of these methods if we do not include regularization.

 \begin{figure}[t]
   \begin{center}
    \includegraphics[width=\textwidth]{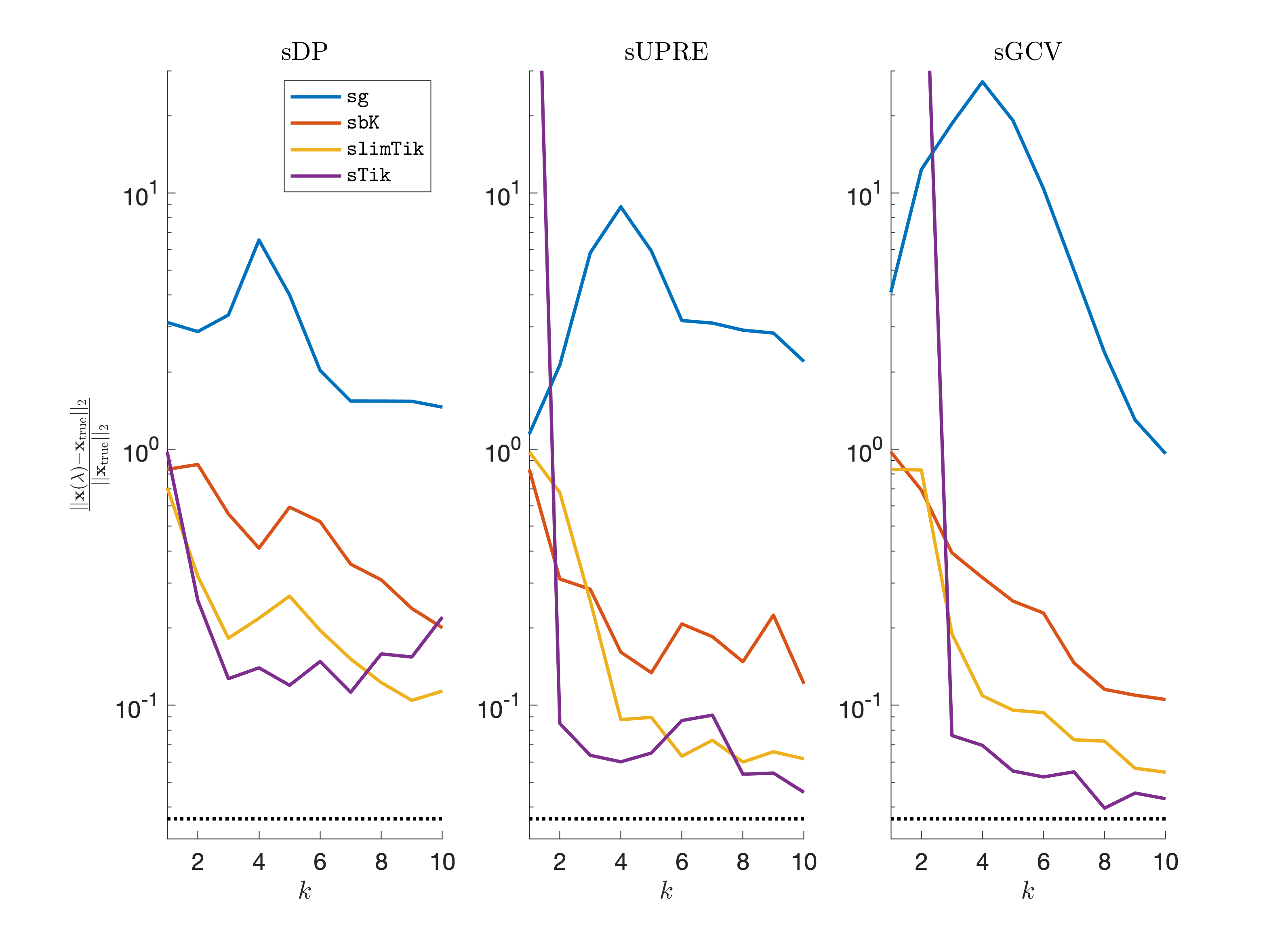}
   \end{center}
   \caption{Comparison of relative reconstruction errors for \texttt{sg}, \texttt{sbK}, \texttt{slimTik}, and \texttt{sTik} iterates for {\tt gravity} using various sampled regularization parameter selection methods. We compare sDP, sUPRE, and sGCV.  The horizontal black line is the relative error corresponding to the optimal regularization parameter for the full problem, which is not feasible to get in practice. } \label{fig:example3}
 \end{figure}
 Having demonstrated that regularization parameter update methods can be incorporated in a variety of stochastic optimization methods, we next investigate the performance of these limited-memory methods for super-resolution image reconstruction.
 The basic goal of super-resolution imaging is to reconstruct an $n \times n$  high-resolution image represented by a vector $\bfx_{\text{true}} \in \mathbb{R}^{n^2}$
 given $M$ low-resolution images of size $\ell \times \ell$ represented by
 $\bfb_1 \cdots, \bfb_{M},$ where $\bfb_{i} \in \mathbb{R}^{\ell^2}\,.$
 The forward model for each low-resolution image is given as
  $$\bfb_{i} = \bfR\bfS_{i} \bfx_{\text{true}} + \bfepsilon_i\,,$$
 where $\bfR \in \mathbb{R}^{\ell^{2} \times n^{2}}$  is a restriction matrix, $\bfS_{i} \in \mathbb{R}^{n^{2}\times n^{2}}$ represents an affine transformation that may account for shifts, rotations, and scalar multiplications, and
 $ \bfepsilon_{i} \sim \mathcal{N}\left({\bf{0}}_{\ell^{2}}, \sigma^{2}{\bfI}_{\ell^{2}} \right)$. To reconstruct a high-resolution image, we solve the Tikhonov problem,
 \begin{equation*}
   \min_{\bfx} \norm[2]{\left[\begin{array}{c}\bfR\bfS_{1} \\ \vdots \\ \bfR\bfS_{M} \end{array}\right] \bfx-\left[\begin{array}{c} \bfb_{1} \\ \vdots\\ \bfb_{M} \end{array}\right]}^{2} + \lambda\norm[2]{\bfL\bfx}^{2}.
 \end{equation*}
 For cases where the low-resolution images are being streamed or where the number of low-resolution images is very large, standard iterative methods may not be feasible.  Furthermore, it can be very challenging to determine a good choice of $\lambda$ prior to solution computation \cite{chung2006numerical,huang2008three,park2003super}.
 \begin{figure}[b]
   \begin{center}
     \begin{tabular}{cc}
     \includegraphics[width=0.66\textwidth]{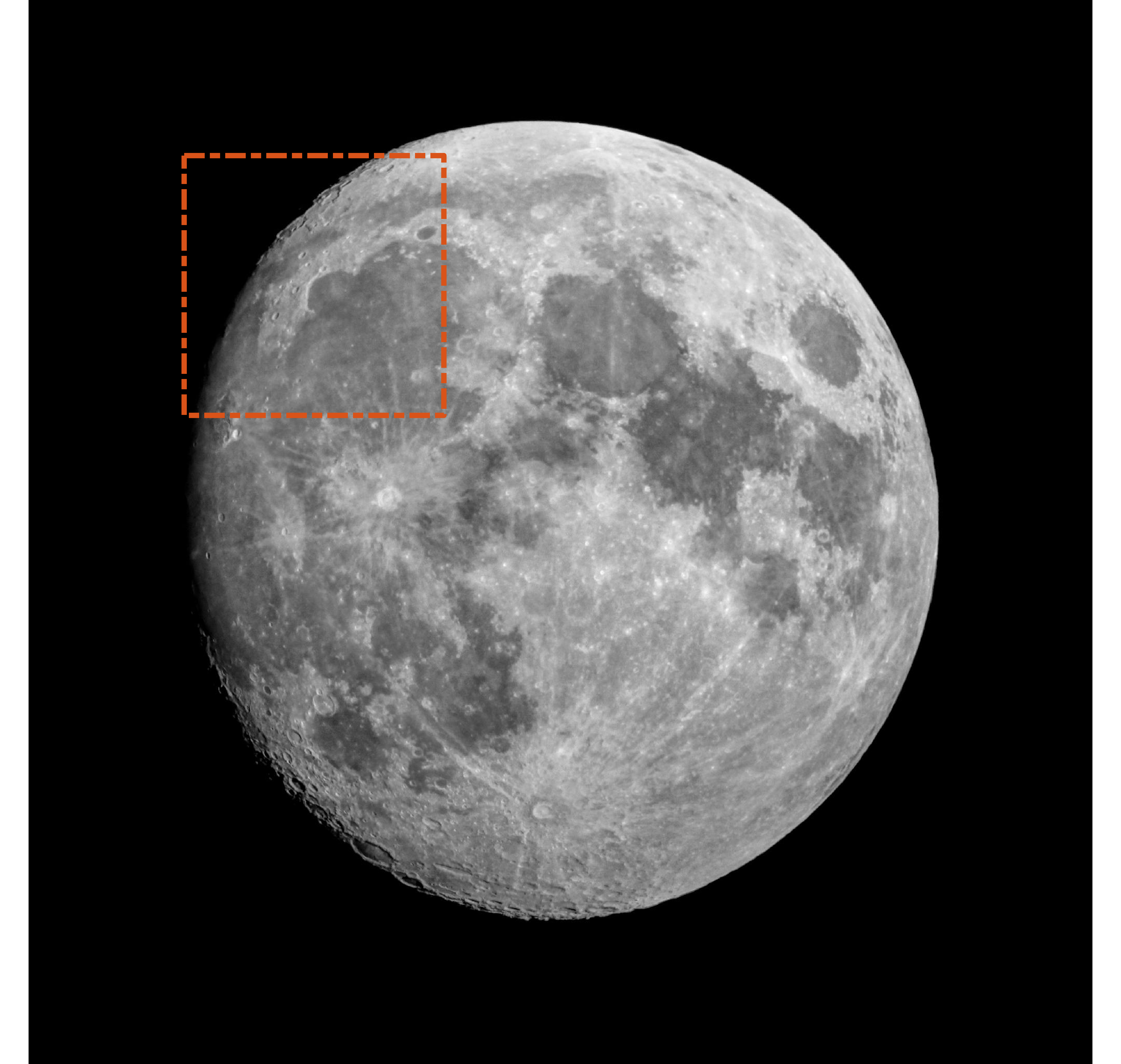} & \hspace*{-4ex}
       \begin{minipage}[H]{0.262\textwidth}
         \vspace*{-52ex}
           \begin{tabular}{c}
             \includegraphics[width=\textwidth]{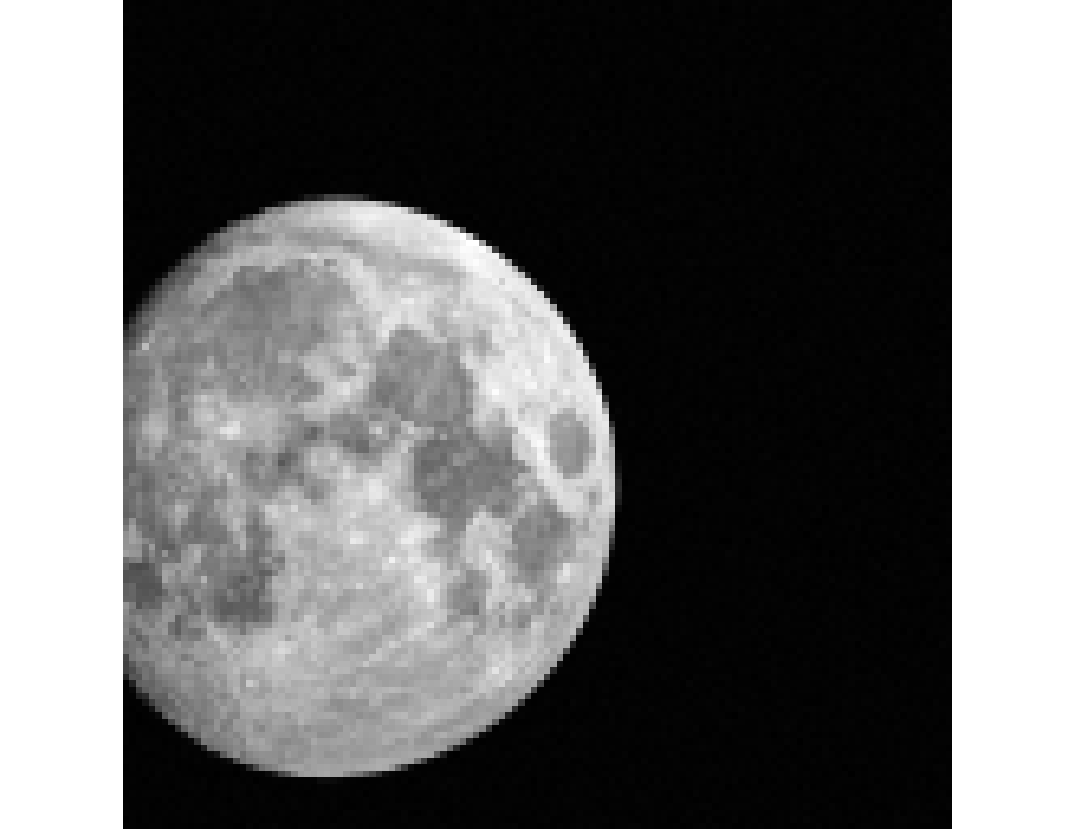}\\
             \includegraphics[width=\textwidth]{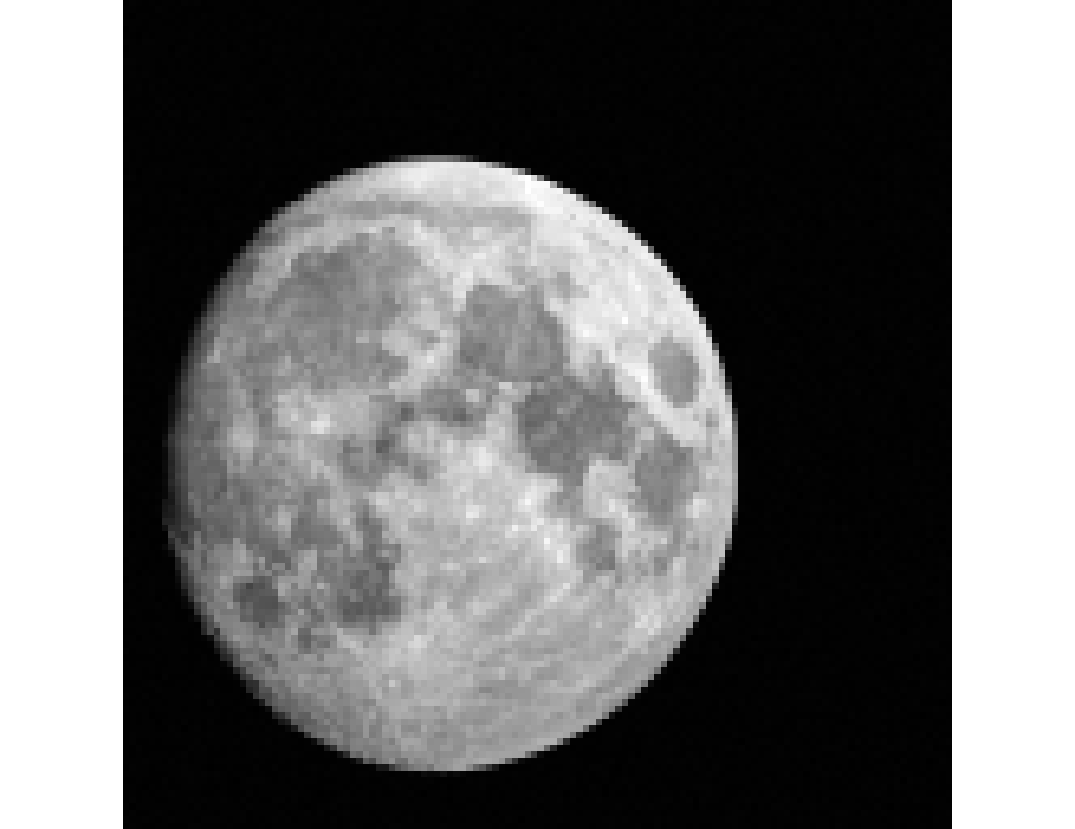}\\
             \includegraphics[width=\textwidth]{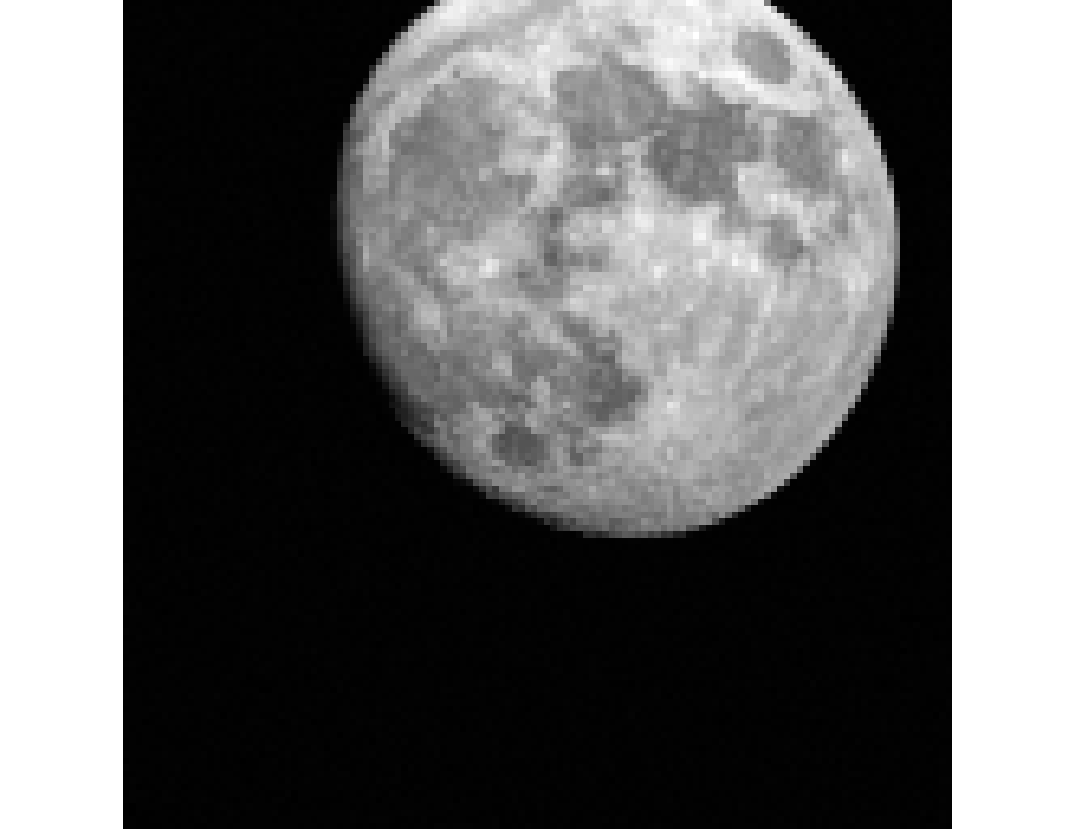}
           \end{tabular}
         \end{minipage}
   \end{tabular}
   \end{center}
   \caption{Super-resolution imaging example.  On the left is the true high-resolution image, and on the right are three sample low-resolution images. The red-box corresponds to sub-images shown in Figure \ref{fig:SRrecon}.}\label{fig:SRimages}
 \end{figure}

 For our example, we have $30$ images of size $128 \times 128$, and we wish to reconstruct a high-resolution image of size $2,\!048 \times 2,\!048$.  In Figure~\ref{fig:SRimages}, we provide the true high-resolution image, which is an image of the moon \cite{NASA14}, and three of the low-resolution images.
 Here, $\bfA_{i} = \bfR\bfS_{i} \in \bbR^{128^2 \times 2,048^2}$.
 Due to the inherent partitioning of the problem, we take $\bfW\t_i \in\bbR^{128^2 \times 30 \cdot 128^2}$ to be a matrix such that $\bfW\t_{i}\bfA= \bfA_{i}$; certainly these $\bfW_{i}$ matrices are never computed.
 For the simulated low-resolution images, Gaussian white noise was added such that the noise level for each image is $0.01$ and take $\bfL=\bfI_{2,048}$.

 We compare the performances of \texttt{sg}, \texttt{sbK}, and \slimTik, including our sampled regularization parameter update methods sDP, sUPRE, and sGCV.
 The true noise variance is used for sDP and sUPRE, and the memory parameter for \slimTik is $r = 2$.
 Each iteration of \texttt{sbK} and \slimTik requires a linear solve, which can be handled efficiently by reformulating the problem as a least squares problem as in equation \eqref{eq:slimTik}, and using standard techniques such as LSQR \cite{PaSa82b,paige1982lsqr}. These iterative methods can also be used to update the regularization parameter. Furthermore, we use the Hutchison trace estimator to efficiently evaluate the trace term in sGCV and sUPRE, see~\eqref{eqn:sampledUPRE} and~\eqref{eq:sampledGCV}.  More specifically, rather than compute $128^2$ linear solves, we note that if $\bfv$ is a random variable such that $\mathbb{E} \, \bfv\bfv\t = \bfI_{128^2}$ , then
 \begin{equation*}
   \trace{\bfW_{\tau(k)}\t \bfA\bfC_k(\lambda)\bfW_{\tau(k)}} = \bbE \bfv\t\bfW_{\tau(k)}\t \bfA\bfC_k(\lambda)\bfW_{\tau(k)}\bfv.
 \end{equation*}
 Here we use the Rademacher distribution where the entries of $\bfv$ are $v_i = \pm 1$ with equal probability.
  We use a single realization of $\bfv$ to approximate the trace, hence resulting in just one linear solve \cite{haber2012effective, avron2011randomized, saibaba2017randomized}.

 Relative reconstruction errors are provided in Figure~\ref{fig:example4}, and sub-images of the reconstructions are provided in Figure~\ref{fig:SRrecon}.
 We observe that, in general, sDP errors perform more erratically compared to sUPRE and sGCV. Notice that for sUPRE and sGCV, \texttt{sbK} produces higher reconstruction errors compared to \texttt{sg}, which may be attributed to insufficient global curvature information. Furthermore, we observe that \slimTik reconstructions contain more details than \texttt{sg} and \texttt{sbK} reconstructions and reconstructions without regularization may become contaminated with noise.

 \begin{figure}[t]
    \begin{center}
     \includegraphics[width=\textwidth]{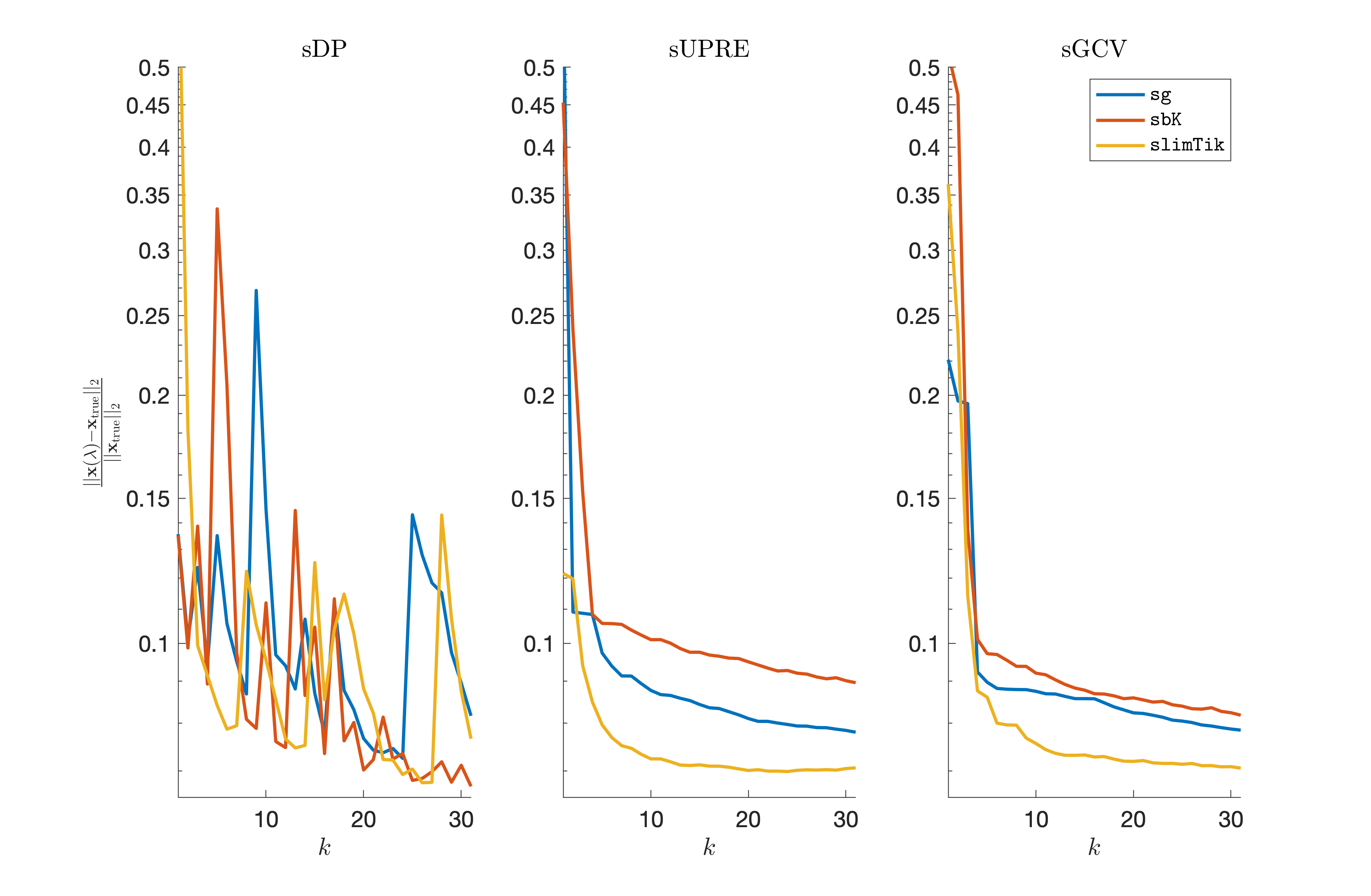}
    \end{center}
    \caption{Relative reconstruction errors for the super-resolution imaging example for one epoch. We note that sUPRE and sGCV produce good reconstructions. Additionally, \slimTik produces a smaller relative reconstruction error, since it is using more curvature information.}\label{fig:example4}
  \end{figure}
 \begin{figure}[t]
   \begin{center}
     \setlength\tabcolsep{0.5pt}
     \begin{tabular}{ccccc}
        & sDP & sUPRE & sGCV &  none\\
       \rotatebox{90}{\hspace*{8ex}{\tt sg}} &
       \includegraphics[width=0.245\textwidth]{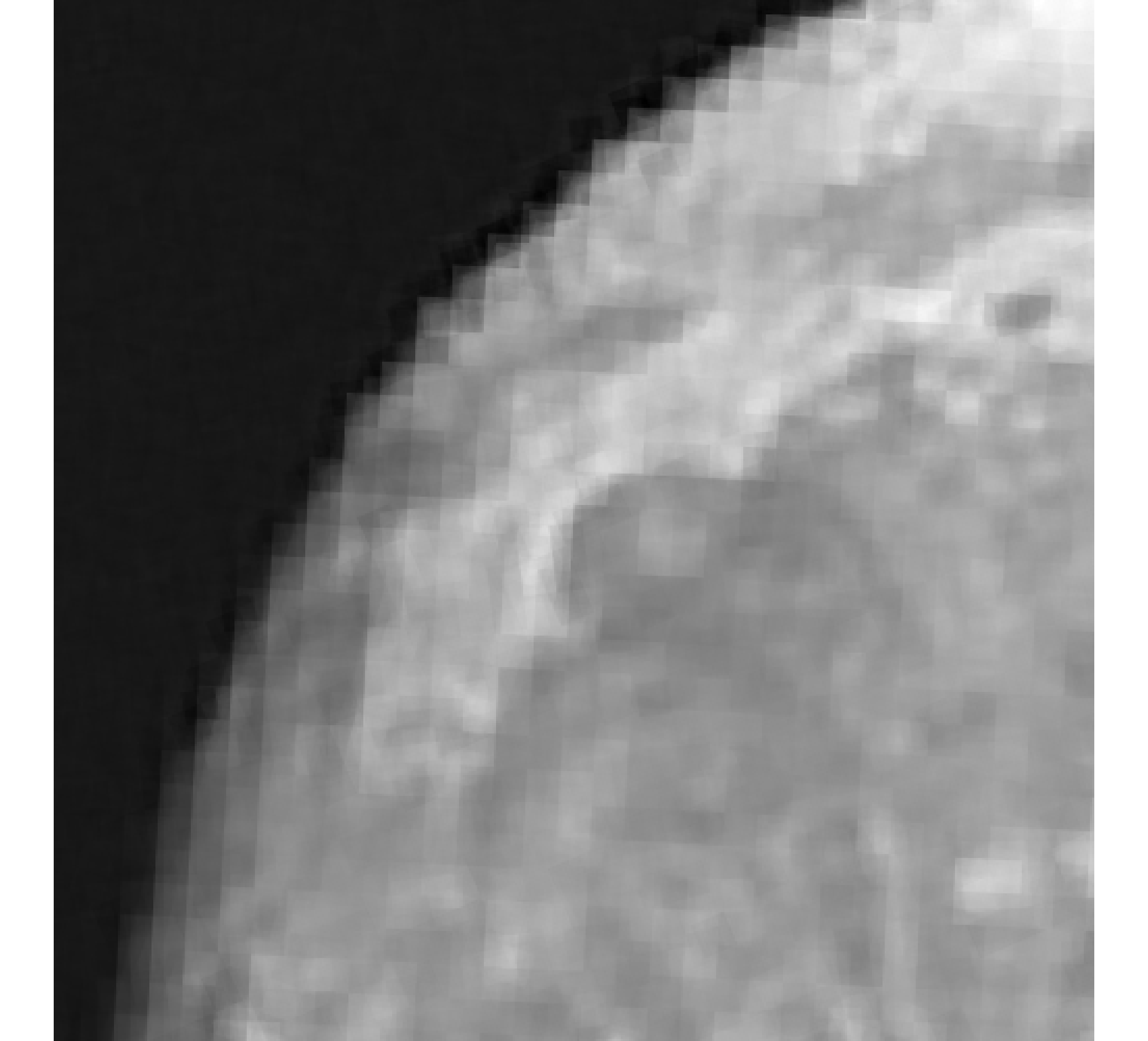} &
       \includegraphics[width=0.245\textwidth]{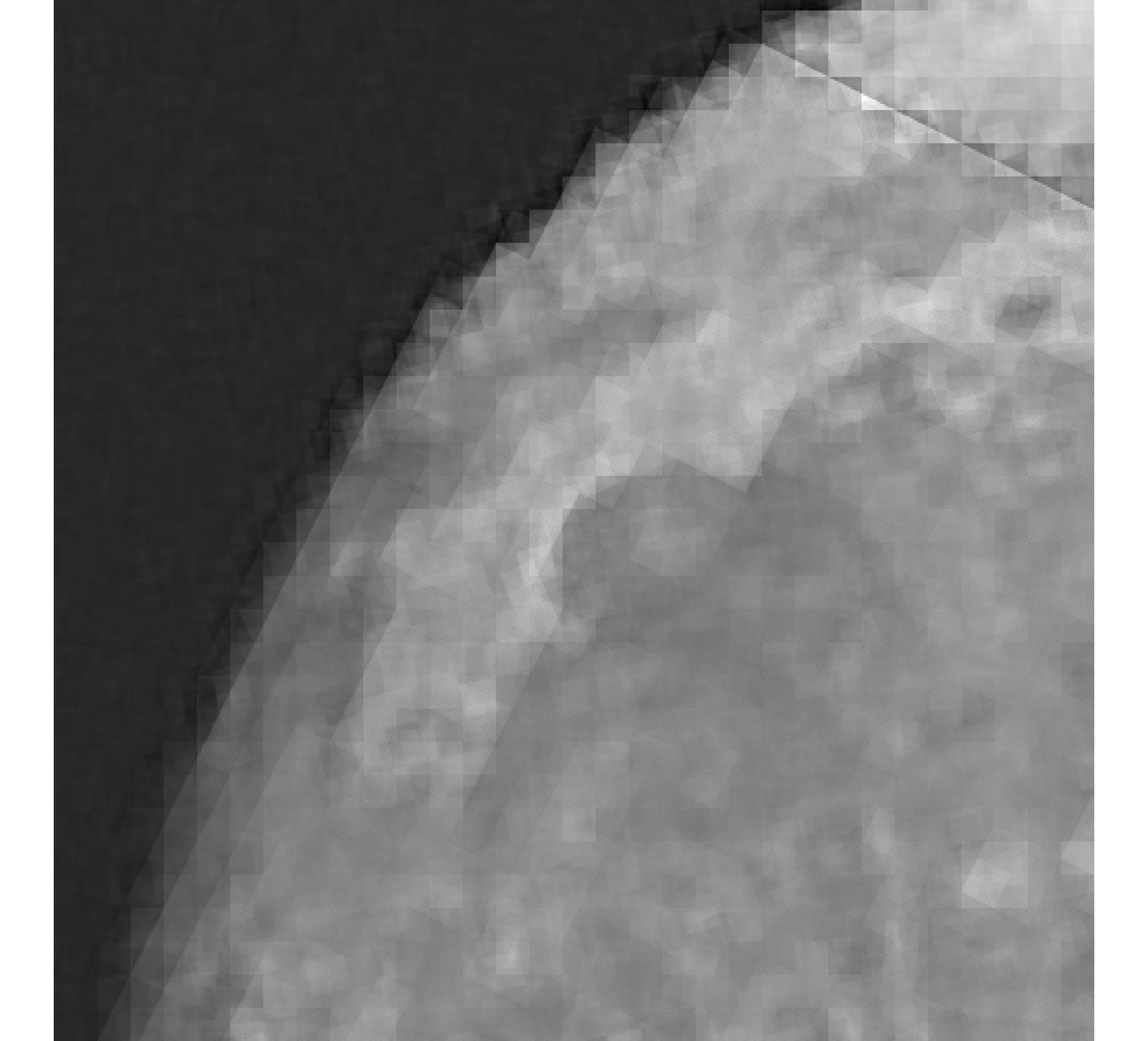} &
       \includegraphics[width=0.245\textwidth]{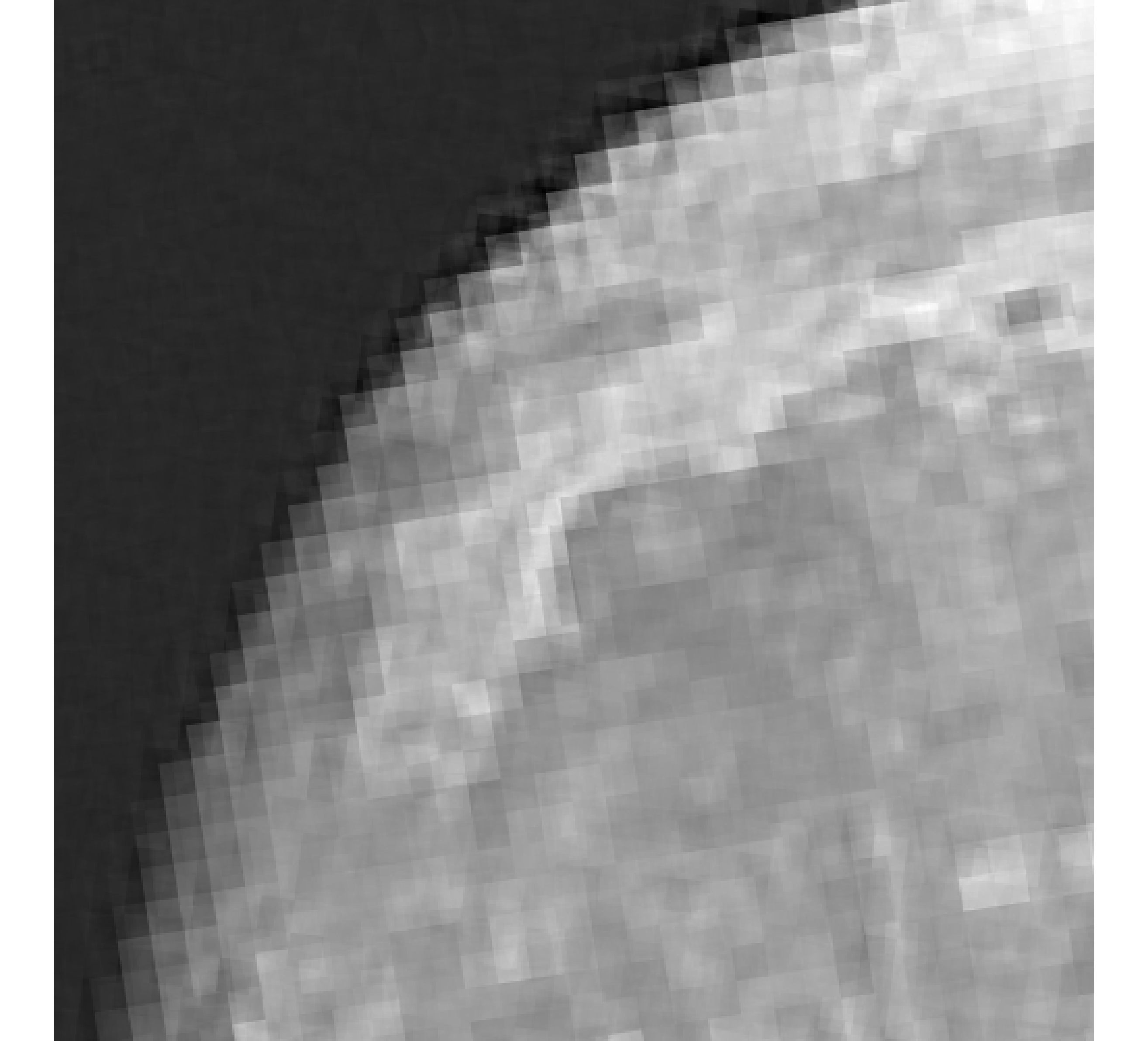} &
       \includegraphics[width=0.245\textwidth]{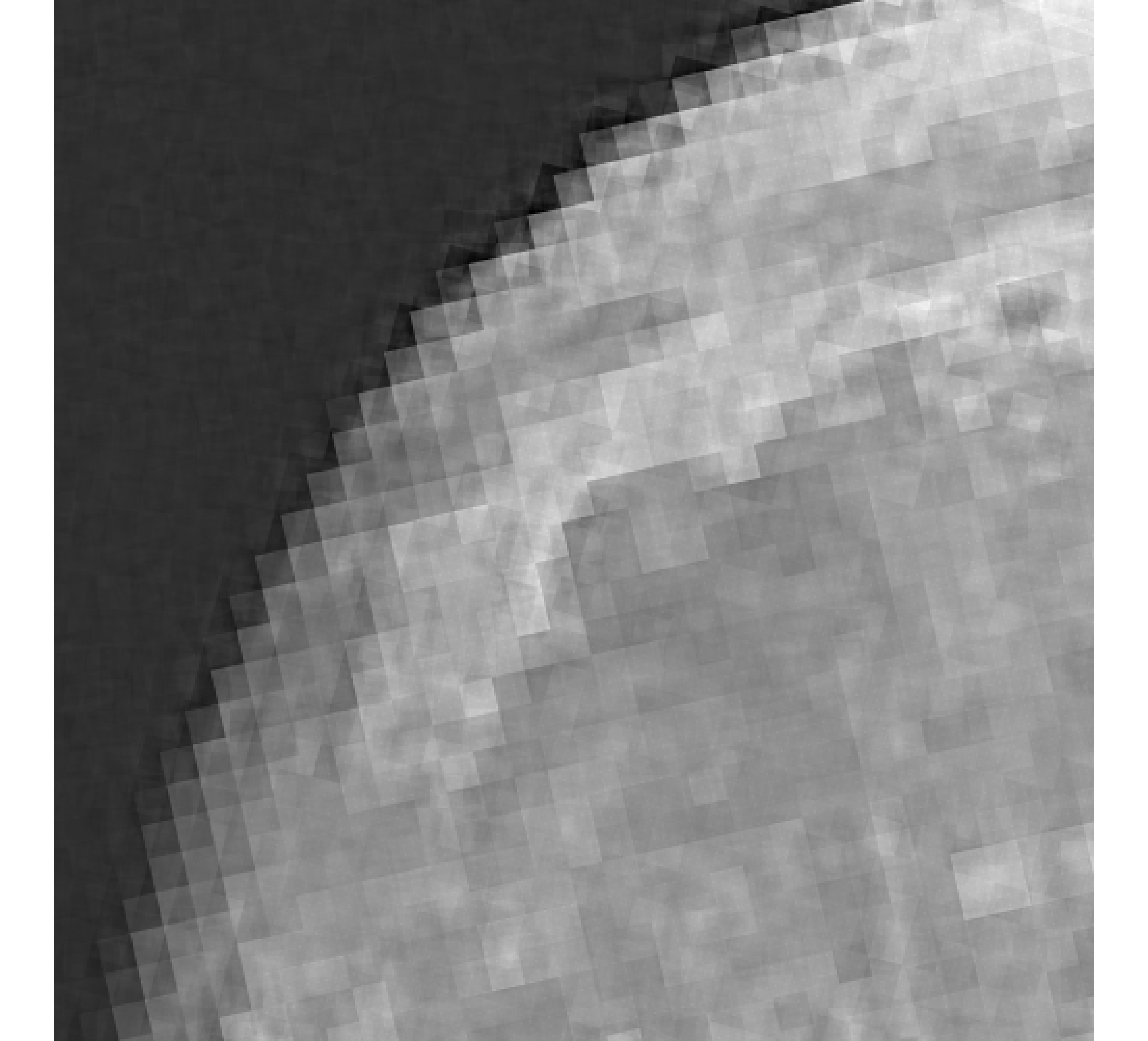} \\
       \rotatebox{90}{\hspace*{8ex}{\tt sbK}} &
       \includegraphics[width=0.245\textwidth]{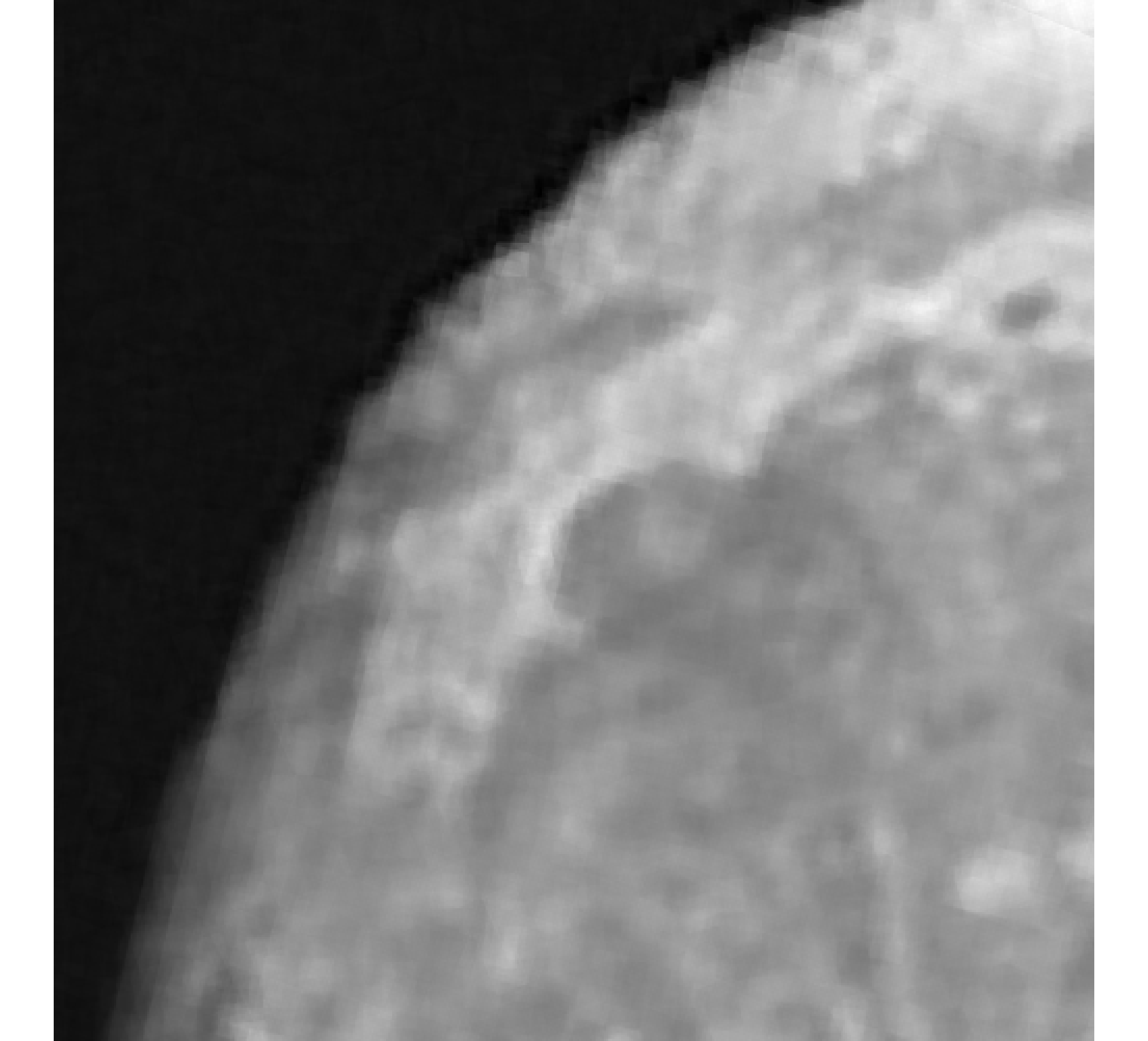} &
       \includegraphics[width=0.245\textwidth]{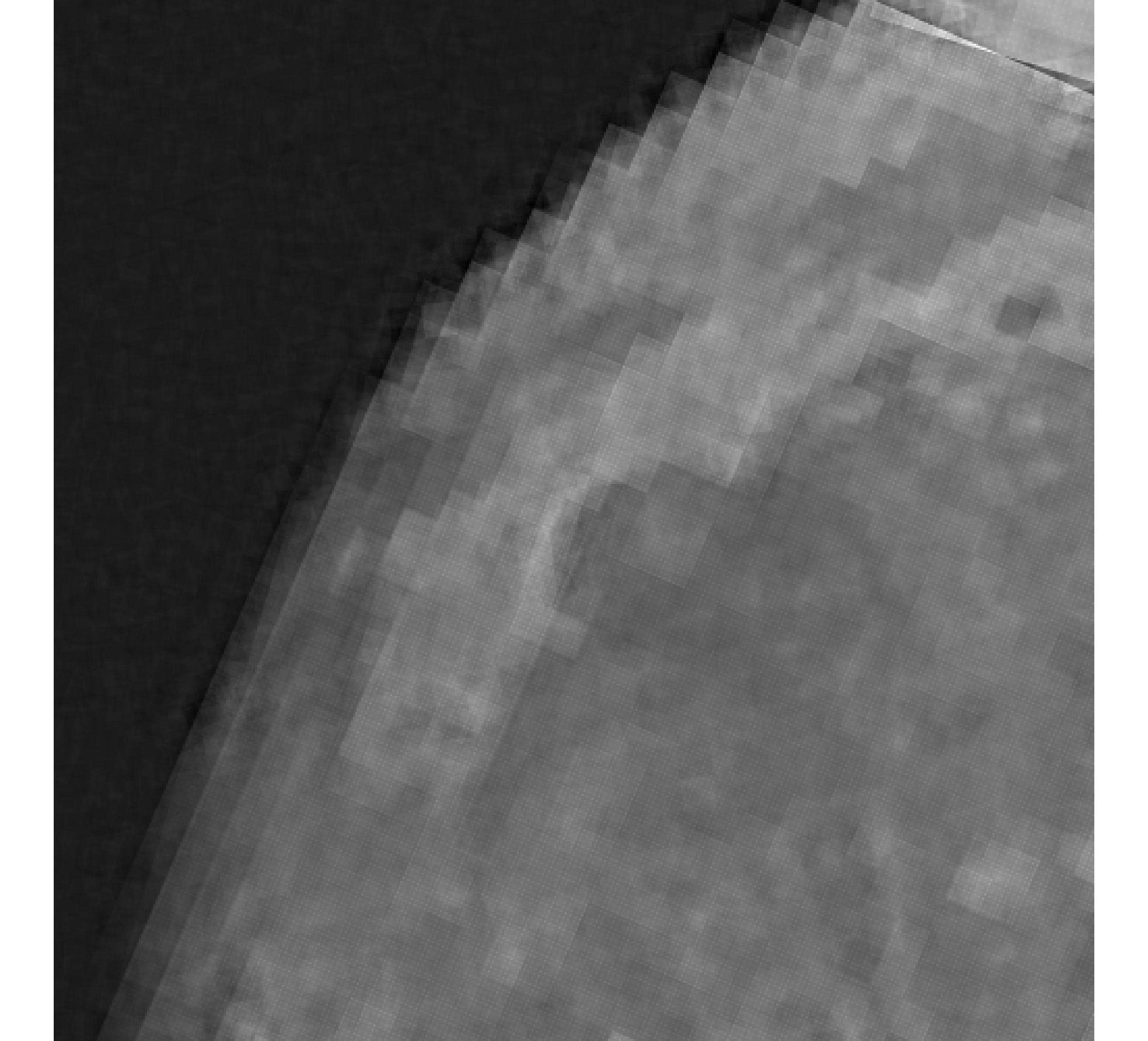} &
       \includegraphics[width=0.245\textwidth]{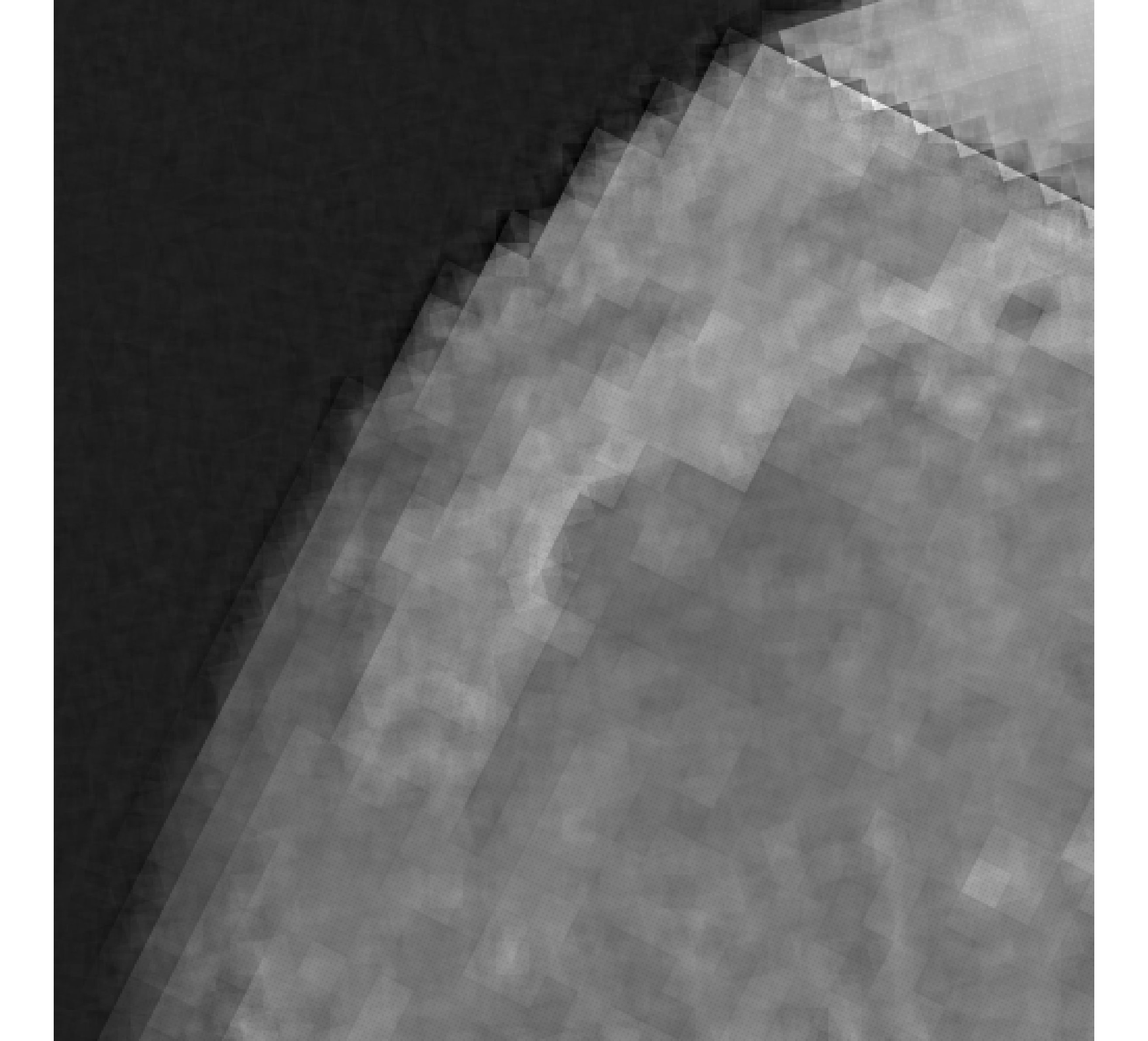} &
       \includegraphics[width=0.245\textwidth]{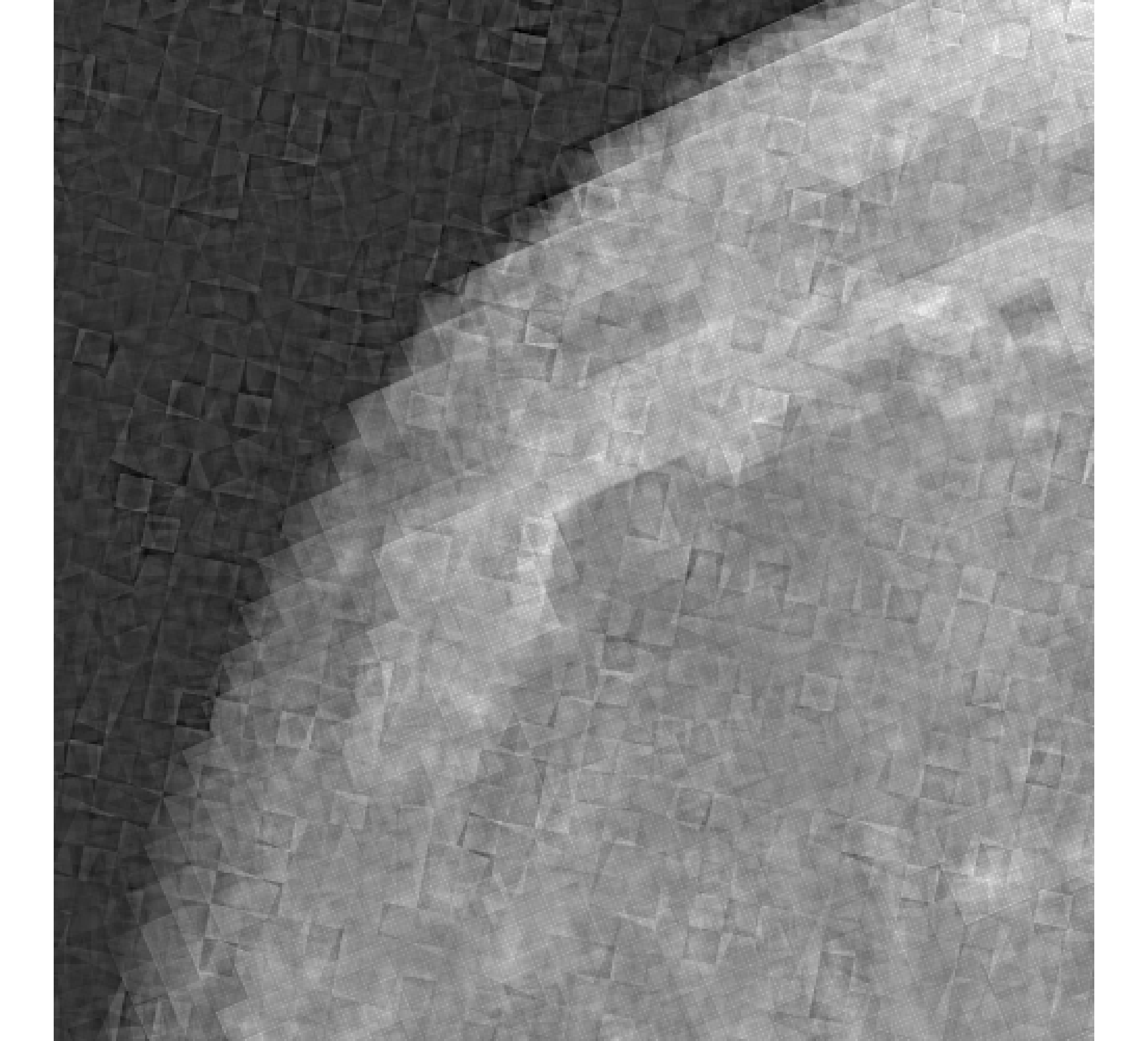} \\
       \rotatebox{90}{\hspace*{5ex}{\tt slimTik}} &
       \includegraphics[width=0.245\textwidth]{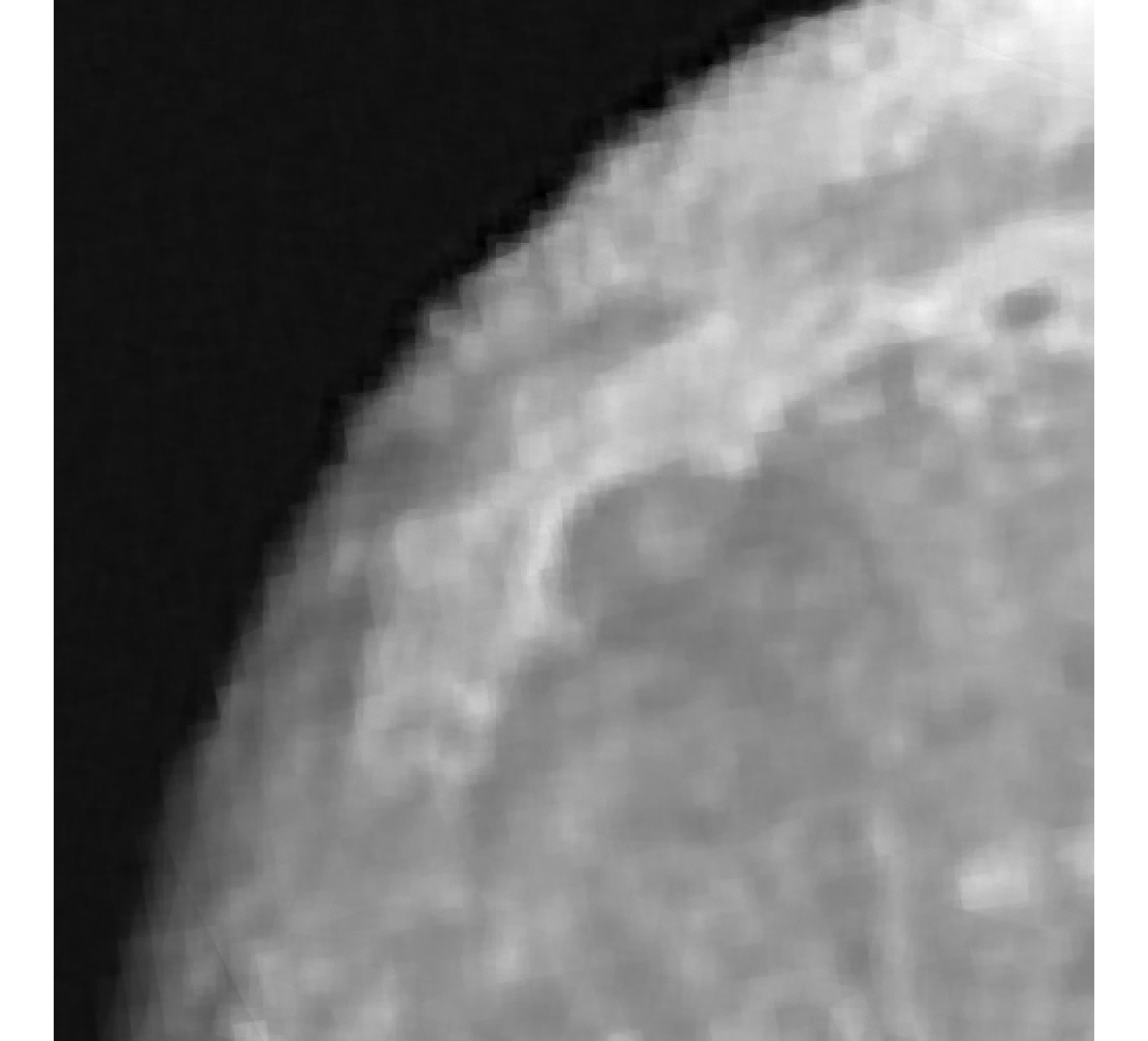} &
       \includegraphics[width=0.245\textwidth]{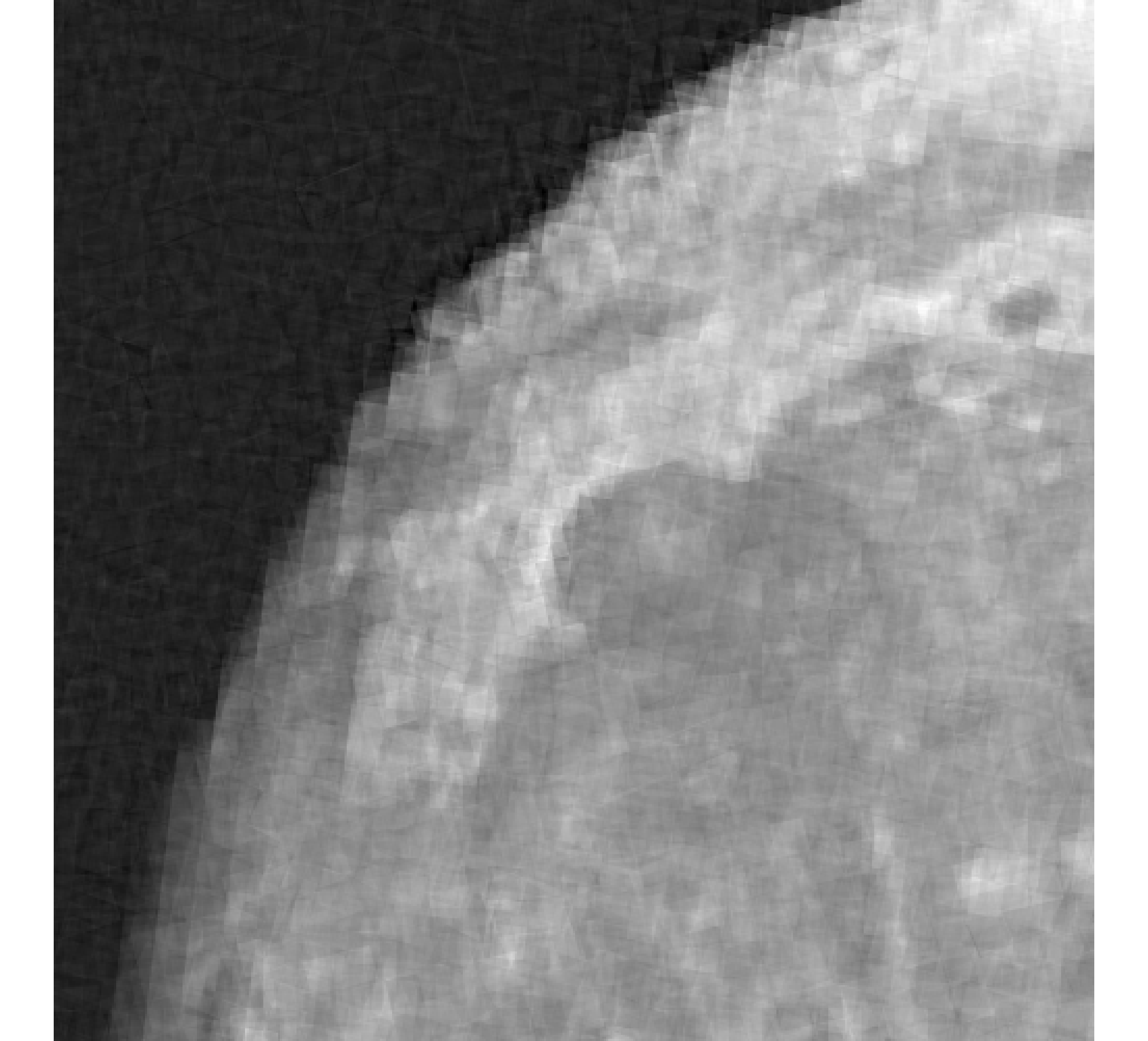} &
       \includegraphics[width=0.245\textwidth]{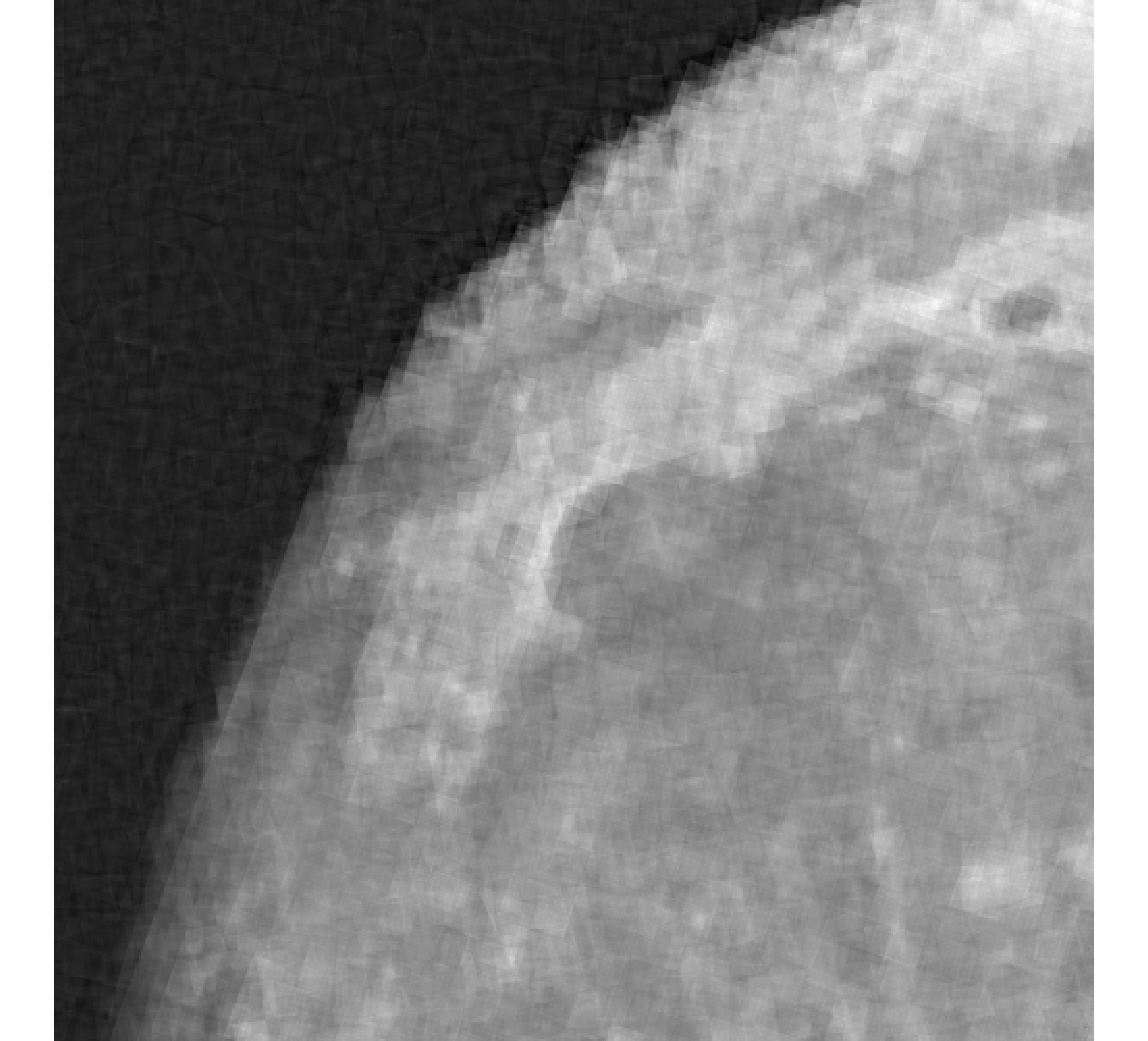} &
       \includegraphics[width=0.245\textwidth]{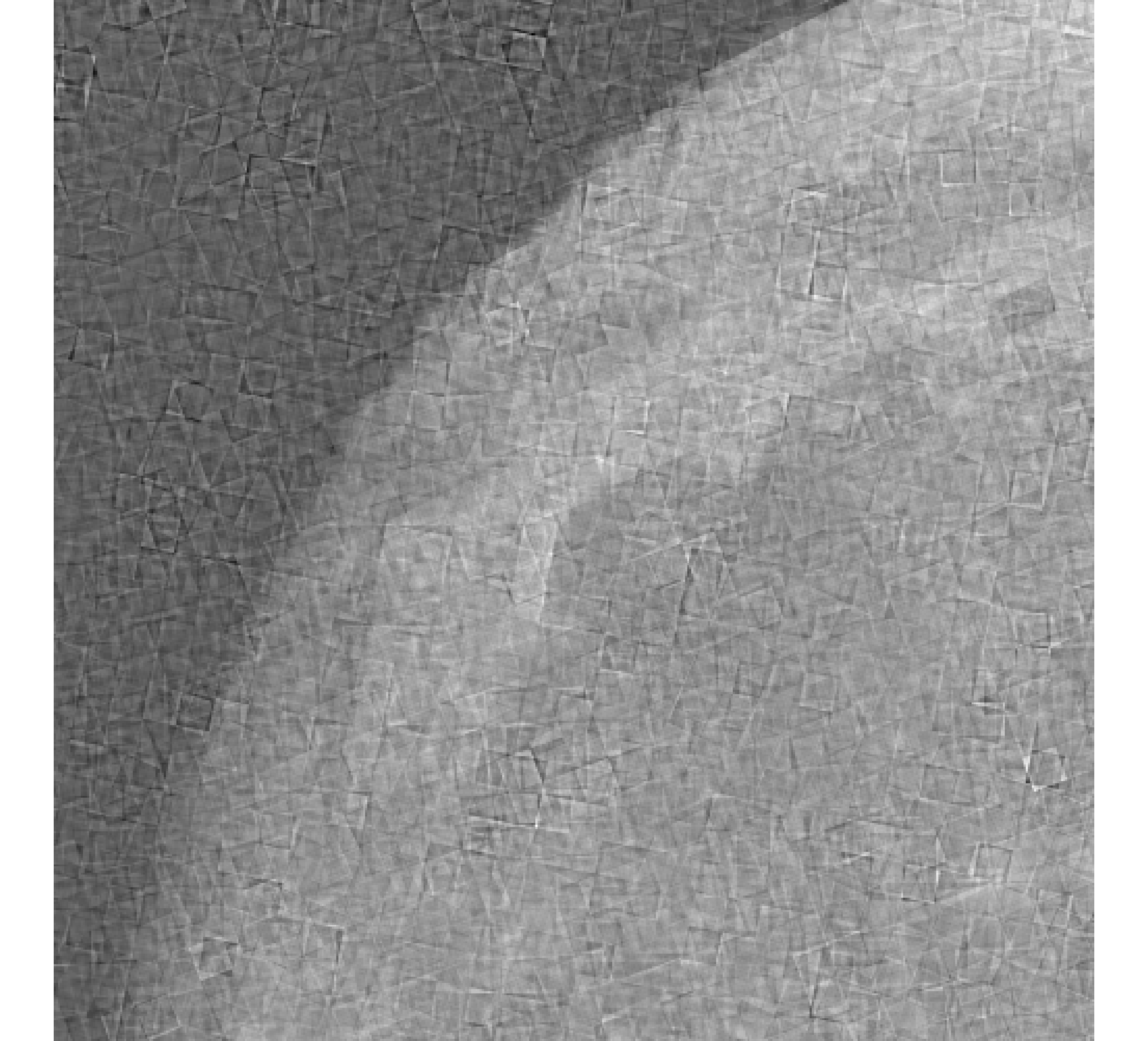} \\
     \end{tabular}
   \end{center}
   \caption{Sub-images of the reconstructed images for the super-resolution imaging example.  Reconstructions correspond to {\tt sg}, {\tt sbK}, and {\tt slimTik} with regularization parameter updates computed using sDP, sUPRE, and sGCV.  For comparison, we provide reconstructions corresponding to no regularization, i.e., $\lambda=0$.}
   \label{fig:SRrecon}
 \end{figure}

 \section{Conclusions}

 \label{sec:conclusions}
 In this work we describe iterative sampled Tikhonov methods for solving inverse problems for which it is not feasible to access the data ``all-at-once''. Such methods are necessary when handling data sets that do not fit in memory and also can naturally handle streaming data or ``online'' problems.

 We investigate two iterative methods, \rrls and \sTik, and show that under various sampling schemes, \rrls iterates converge asymptotically to the unregularized solution, while \sTik iterates converge to a Tikhonov-regularized solution.  Although the sampling mechanisms that we discuss do not play a role in the asymptotic convergence, they do allow for interesting interpretations.  In particular, for random cyclic sampling we can characterize the iterates as Tikhonov solutions after every epoch, providing insight into the path that the iterates take towards the solution. Note that this primarily applies to the massive-scale inverse problems where we have the opportunity to choose a sampling method.
 Furthermore, for iterative methods where the regularization parameter can be updated during the iterative process (e.g., \sTik),
 we describe sampled variants of existing regularization parameter selection methods to update the parameter.
 Using a number of well-known data sets, we show empirically that sampled Tikhonov methods with automatic regularization parameter updates can be competitive.
 For very large inverse problems, we describe a limited-memory version of \sTik, and we demonstrate the efficacy of the limited-memory approach on a standard benchmark dataset as well as on a streaming super-resolution image reconstruction problem.

 Future directions of research include developing an asymptotic analysis of \slimTik and a non-asymptotic analysis of the general sampling algorithms.  This would involve bounding the mean square error at a fixed iteration $k$, which may help to explain the quick initial convergence seen in the numerical experiments. Another open question is how to accelerate convergence by selecting $\bfW_{\tau}$ to sample ``important" parts of the problem, e.g., using sketching matrices \cite{drineas2012fast, drineas2011faster}.
 Finally, extensions to nonlinear inverse problems would require more advanced convergence analyses and further algorithmic developments, e.g., incorporating adaptive regularization parameter selection within stochastic LBFGS \cite{Byrd2016, mokhtari2015global}.

 \appendix
 \section{Proofs for Section~\ref{sec:rrls}} \label{sec:appendix1}

 \begin{proof}[Proof of Theorem \ref{thm:sqnrls}]
 For (ii), note that the solution of the least squares problem~\eqref{eq:lskl} is given by
   \begin{equation*}
     \bfx(\lambda_k) = \bfB_k \sum_{i=1}^{k} \bfA_{i}\t  \bfb_{i}.
   \end{equation*}
   Noticing the relationship $\bfB_k^{-1} = \bfB_{k-1}^{-1} + \bfA_{k}\t \bfA_{k} + \Lambda_k\bfL\t\bfL$, we get the following equivalencies for the \sTik iterates
   \begin{align*}
     \bfx_k &= \bfx_{k-1} - \bfB_k\left(\bfA_{k}\t(\bfA_{k}\bfx_{k-1}-\bfb_{k})+\Lambda_k\bfL\t\bfL\bfx_{k-1}\right)\\
              &= \bfB_k\left( \bfB_k^{-1}\bfx_{k-1} - \bfA_{k}\t \bfA_{k}\bfx_{k-1}+\bfA_{k}\t \bfb_{k} -\Lambda_k\bfL\t\bfL\bfx_{k-1}\right)\\
              &= \bfB_k\left( \bfB_{k-1}^{-1}\bfx_{k-1} +\bfA_{k}\t \bfb_{k}\right)
              = \bfB_k\sum_{i = 1}^{k}\bfA_{i}\t \bfb_{i} = \bfx(\lambda_k).
   \end{align*}
 A similar proof can be made for (i).
 \end{proof}

 \begin{proof}[Proof of Theorem \ref{thm:converge_xls}]
 \begin{enumerate} \item
   From Theorem \ref{thm:sqnrls} for any $k\in \bbN$ we have
   \begin{align*}
     \bfy_k  & = \left(\lambda\bfL\t\bfL + \sum_{i=1}^k \bfA\t\bfW_{\tau(i)}\bfW_{\tau(i)}\t\bfA\right)^{-1}\left(\sum_{i=1}^k\bfA\t\bfW_{\tau(i)} \bfW_{\tau(i)}\t\bfb + \lambda\bfL\t\bfL\bfy_0 \right) \\
     & = \left(\frac{\lambda\bfL\t\bfL + \sum_{i=1}^k \bfA\t\bfW_{\tau(i)}\bfW_{\tau(i)}\t\bfA}{k}\right)^{-1}\left(\frac{\sum_{i=1}^k \bfA\t\bfW_{\tau(i)}\bfW_{\tau(i)}\t\bfb + \lambda\bfL\t\bfL\bfy_0 }{k}\right).
   \end{align*}
   Using the fact that $\bbE \, \bfW_{\tau(i)}\bfW_{\tau(i)}\t = \tfrac{\ell}{m}\bfI_m$ (see equation~\eqref{eq:Wexpectation}), by the law of large numbers and Slutsky's theorem for a.s. convergence \cite{tenorio2017introduction}
         \begin{align*}
       \frac{\sum_{i=1}^k \bfA\t\bfW_{\tau(i)}\bfW_{\tau(i)}\t\bfb + \lambda\bfL\t\bfL\bfy_0 }{k}  & \arras  \frac{\ell}{m}\bfA\t\bfb,
       \end{align*}
      and
       \begin{equation*}
         \left(\frac{\lambda\bfL\t\bfL + \sum_{i=1}^k \bfA\t\bfW_{\tau(i)}\bfW_{\tau(i)}\t\bfA}{k}\right)^{-1}\arras \frac{m}{\ell}\left(\bfA\t\bfA\right)^{-1}.
       \end{equation*}
 and therefore
   \begin{equation*}
     \bfy_k \arras \left( \bfA\t\bfA \right)^{-1}\bfA\t\bfb = \bfx(0).
   \end{equation*} $\left.\right.$
 \item
   In a similar fashion, for any $k\in \bbN$ we have
 \begin{equation*}
   \bfx_k  = \left(\frac{ \sum_{i=1}^k \Lambda_i\bfL\t\bfL + \bfA\t\bfW_{\tau(i)}\bfW_{\tau(i)}\t\bfA}{k}\right)^{-1}\left(\frac{\sum_{i=1}^k \bfA\t\bfW_{\tau(i)}\bfW_{\tau(i)}\t\bfb}{k}\right).
 \end{equation*}
    Using the fact that $\bbE \, \bfW_{\tau(i)}\bfW_{\tau(i)}\t = \tfrac{\ell}{m}\bfI_m$ and $\lim_{k\to\infty}\frac{\sum_{i=1}^k \Lambda_i\bfL\t\bfL}{k} = \frac{\ell}{m}\lambda\bfL\t\bfL$,
    we have
       \begin{equation*}
         \frac{\sum_{i=1}^k \bfA\t\bfW_i\bfW_i\t\bfb}{k}\arras \frac{\ell}{m}\bfA\t\bfb
       \end{equation*}
  and
   \begin{equation*}
         \left(\frac{\sum_{i=1}^k \Lambda_i\bfL\t\bfL +  \bfA\t\bfW_i\bfW_i\t\bfA}{k}\right)^{-1}\arras \frac{m}{\ell}\left(\bfA\t\bfA + \lambda \bfL\t\bfL\right)^{-1},
       \end{equation*}
   and thus we conclude that \begin{equation*}
     \bfx_k \arras \left(\bfA\t\bfA + \lambda \bfL\t\bfL\right)^{-1} \bfA\t\bfb = \bfx(\lambda).
   \end{equation*}
   \end{enumerate}
 \end{proof}

 \begin{proof}[Proof of Theorem \ref{thm:scwr}]
 Notice that for random cyclic sampling schemes and for any iteration $jM$, $\sum_{i =1}^{jM}\bfA\t\bfW_{\tau{(i)}}\bfW_{\tau{(i)}}\t\bfA = j \bfA\t\bfA$ and $\sum_{i =1}^{jM}\bfA\t\bfW_{\tau{(i)}}\bfW_{\tau{(i)}}\t\bfb = j  \bfA\t\bfb$ are deterministic. Hence
   $$
     \bfy_{jM} = j  \left( \lambda \bfL\t\bfL + j  \bfA\t\bfA \right)^{-1}  \bfA\t\bfb = \left( \frac{\lambda}{j } \bfL\t\bfL + \bfA\t\bfA \right)^{-1} \bfA\t\bfb = \bfx\left(\tfrac{1}{j}\lambda \right)
   $$
 and
   $$  \bfx_{jM} = j  \left( \lambda_{jM} \bfL\t\bfL + j \bfA\t\bfA \right)^{-1}  \bfA\t\bfb = \left( \frac{\lambda_{jM}}{j } \bfL\t\bfL + \bfA\t\bfA \right)^{-1} \bfA\t\bfb =
      \bfx\left(\tfrac{1}{j}\lambda_{j M}\right).
 $$
 \end{proof}

 \section{Derivations for sampled UPRE and sampled GCV}\label{sec:regparams}
 In this section, we provide derivations for \eqref{eqn:sampledUPRE} and \eqref{eq:sampledGCV}. To estimate the overall regularization parameter $\lambda$ at the $k$-th iteration we are just required to update $\Lambda_k$ since the estimate $\lambda$ is uniquely determined by the preceding $\Lambda_i$'s, $i = 1,\ldots,k-1$ and $\Lambda_k$. Hence, for ease of notation we will drop the iteration count on $\lambda$.

 \subsection{Derivation of the Sampled UPRE}
 \label{sec:SUPRE}
 The basic idea is to find $\Lambda_k$ by minimizing an estimate of the predictive error. Let the \emph{sampled predictive error} be given by
 \begin{equation*}
   P(\lambda) = \norm[2]{\bfW_{\tau(k)}\t(\bfA \bfx_k(\lambda) -\bfA\bfx_{\true})}^2.
 \end{equation*}
 Using the notation from \eqref{eqn:regularized}, the expected sampled predictive error, $\bbE\, P(\lambda)$, can be written as
 \begin{equation}
   \label{eq:derivE}
  \bbE  \norm[2]{\bfW_{\tau(k)}\t\left(\bfA \bfC_k(\lambda)-\bfI_m\right)\bfA \bfx_{\true}}^2 + \sigma^2  \bbE\,\trace{ \bfC_k(\lambda)\t\bfA\t\bfW_{\tau(k)} \bfW_{\tau(k)}\t\bfA\bfC_k(\lambda)},
 \end{equation}
 where the mixed term vanishes due to independence of $\bfW_{\tau(1)}, \ldots \bfW_{\tau(k)}$ and $\bfepsilon$ and since $\bbE \bfepsilon = \bf0$.  Similar to the derivation for standard UPRE, the predictive error is not computable in practice since $\bfx_{\true}$ is not available.  Thus, we perform a similar calculation for the expected sampled residual norm,
 \begin{align}
   \bbE &\ \norm[2]{\bfW_{\tau(k)}\t\left(\bfA \bfx_k(\lambda) -\bfb\right)}^2  =
   \bbE  \norm[2]{\bfW_{\tau(k)}\t(\bfA \bfC_k(\lambda)-\bfI_m)\bfb}^2 \nonumber \\
   & \hspace*{3ex}=
   \bbE  \norm[2]{\bfW_{\tau(k)}\t(\bfA\bfC_k(\lambda)-\bfI_m)\bfA \bfx_{\true}}^2  + \bbE \norm[2]{\bfW_{\tau(k)}\t(\bfA\bfC_k(\lambda)-\bfI_m)\bfepsilon}^2. \label{eq:sampledresUPRE}
 \end{align}
 Next, notice that using the trace lemma for a symmetric matrix \cite{Bardsley2018}, the second term in~\eqref{eq:sampledresUPRE} can be written as
 \begin{equation} \label{eq:derivation1}
 \sigma^2 \Big(\bbE\, \trace{\bfC_k(\lambda)\t\bfA\t \bfW_{\tau(k)} \bfW_{\tau(k)}\t\bfA\bfC_k(\lambda)}  -2\, \bbE\, \trace{\bfW_{\tau(k)} \bfW_{\tau(k)}\t\bfA\bfC_k(\lambda)} + \ell \Big).
 \end{equation}
 Combining~\eqref{eq:derivE} with \eqref{eq:sampledresUPRE} and~\eqref{eq:derivation1}, we get
 \begin{equation*}
   \bbE P(\lambda) = \bbE  \norm[2]{\bfW_{\tau(k)}\t\left(\bfA \bfx_k(\lambda) -\bfb\right)}^2 + 2 \sigma^2\, \bbE\, \trace{\bfW_{\tau(k)} \bfW_{\tau(k)}\t \bfA\bfC_k(\lambda)} - \sigma^2 \ell.
 \end{equation*}
 Finally for a given realization, we get an
  estimator for the predictive risk
 \begin{equation*}
   U_k(\lambda) = \norm[2]{\bfW_{\tau(k)}\t\left(\bfA \bfx_k(\lambda) -\bfb\right)}^2 + 2 \sigma^2 \trace{\bfW_{\tau(k)}\t \bfA\bfC_k(\lambda)\bfW_{\tau(k)}} -  \sigma^2 \ell\,,
 \end{equation*}
 which is equivalent to \eqref{eqn:sampledUPRE}.

 \subsection{Derivation of the Sampled GCV}
 \label{sec:SGCV}
 Next, we derive the sampled generalized cross validation function, following a similar derivation of the cross validation and generalized cross validation function found in \cite{golub1979generalized}. For notational simplicity, we denote $\bfA_{\tau(i)} = \bfW_{\tau(i)}\t \bfA$ and $\bfb_{\tau(i)} = \bfW_{\tau(i)}\t \bfb$.
 Then, notice that the $k$-th iterate of {\tt sTik}, which is given by $\bfx_k(\lambda) = \bfC_k(\lambda)\bfb$
 is the solution to the following problem,
 \begin{equation*}
   \min_\bfx \norm[2]{\bfA_{\tau(k)}\bfx -\bfb_{\tau(k)}}^2 + \lambda \norm[2]{\bfL \bfx}^2 + \norm[2]{\begin{bmatrix}\bfA_{\tau(1)} \\ \vdots\\\bfA_{\tau(k-1)} \end{bmatrix} \bfx - \begin{bmatrix}\bfb_{\tau(1)} \\ \vdots\\\bfb_{\tau(k-1)} \end{bmatrix}}^2\,.
 \end{equation*}

 To derive sampled GCV, at the $k$-th iterate, define the $\ell \times \ell$ identity matrix with $0$ is the $j$-th entry, i.e.,
 \begin{equation*}
   \bfE_j = \bfI_\ell - \bfe_j\t\bfe_j,
 \end{equation*}
 here $\bfe_j$ is the $j$-th column of the identity matrix.  Our goal is to find $\bfx_{[j]}(\lambda)$, which is the solution to
 \begin{equation*}
   \label{eq:xksolution}
   \min_\bfx \norm[2]{\bfE_j\left(\bfA_{\tau(k)}\bfx -\bfb_{\tau(k)}\right)}^2 + \lambda \norm[2]{\bfL \bfx}^2 + \norm[2]{\begin{bmatrix}\bfA_{\tau(1)} \\ \vdots\\\bfA_{\tau(k-1)} \end{bmatrix} \bfx - \begin{bmatrix}\bfb_{\tau(1)} \\ \vdots\\\bfb_{\tau(k-1)} \end{bmatrix}}^2.
 \end{equation*}
 Then, the sampled cross-validation estimate for $\lambda$ minimizes the average error,
 \begin{equation*}
   V_k(\lambda) = \frac{1}{\ell} \sum_{j=1}^\ell \left(\bfe_j\t\bfb_{\tau(k)} - \bfe_j\t \bfA_{\tau(k)}  \bfx_{[j]}(\lambda)\right)^2.
 \end{equation*}
 Using the normal equations and the fact that $\bfE_j\t \bfE_j = \bfE_j$, an explicit expression for $\bfx_{[j]}(\lambda)$ is given as
 \begin{align*}
   \bfx_{[j]}(\lambda) & = \left(\bfA_{\tau(k)}\t \bfE_j\t \bfE_j \bfA_{\tau(k)} + \lambda\bfL\t \bfL + \sum_{i=1}^{k-1}\bfA_{\tau(i)}\t\bfA_{\tau(i)}\right)^{-1} \left( \bfA_{\tau(k)}\t\bfE_j\t \bfE_j \bfb_{\tau(k)} + \sum_{i=1}^{k-1}\bfA_{\tau(i)}\t\bfb_{\tau(i)} \right)\\
   & =\left(\bfB_k(\lambda)^{-1} - \bfA_{\tau(i)}\t\bfe_j \bfe_j\t\bfA_{\tau(i)}\right)^{-1} \left(\sum_{i=1}^{k}\bfA_{\tau(i)}\t\bfb_{\tau(i)} -\bfA_{\tau(k)}\t\bfe_j \bfe_j\t \bfb_{\tau(k)} \right),
 \end{align*}
 where $\bfB_k(\lambda) = \left(\lambda\bfL\t \bfL + \sum_{i=1}^{k}\bfA_{\tau(i)}\t\bfA_{\tau(i)}\right)^{-1}$.
 Next defining $t_{jj} = \bfe_j \t \bfA_{\tau(k)} \bfB_k(\lambda)\bfA_{\tau(k)}\t \bfe_j$ and using the Sherman-Morrison-Woodbury formula, we get
 \begin{equation*}
   \left(\bfB_k(\lambda)^{-1} - \bfA_{\tau(i)}\t\bfe_j \bfe_j\t\bfA_{\tau(i)}\right)^{-1} = \tfrac{1}{1-t_{jj}}\left((1-t_{jj})\bfB_k(\lambda) + \bfB_k(\lambda) \bfA_{\tau(k)}\t \bfe_j \bfe_j\t \bfA_{\tau(k)} \bfB_k(\lambda)\right)
 \end{equation*}
 and after some algebraic manipulations, we arrive at
 \begin{equation*}
   \bfe_j\t\bfA_{\tau(k)}  \bfx_{[j]}(\lambda)  = \frac{1}{1-t_{jj}} \left(\bfe_j\t \bfA_{\tau(k)} \bfC_k(\lambda)\bfb - t_{jj} \bfe_j\t \bfb_{\tau(k)} \right).
 \end{equation*}
 Thus,
 $$\bfe_j\t\bfb_{\tau(k)} - \bfe_j\t \bfA_{\tau(k)}\bfx_{[j]}(\lambda) = \frac{1}{1-t_{jj}} \bfe_j\t \left(\bfb_{\tau(k)} - \bfA_{\tau(k)}\bfx_k(\lambda)\right)$$
 and we can write the sampled cross-validation function as
 $$
   V_k(\lambda) = \frac{1}{\ell}\norm[2]{\bfD_k(\lambda)(\bfb_{\tau(k)} - \bfA_{\tau(k)}\bfx_k(\lambda))}^2,
 $$
 where $\bfD_k(\lambda)=\diag{\frac{1}{1-t_{11}},\ldots, \frac{1}{1-t_{\ell\ell}}}$.
 Now the extension from the sampled cross-validation to the sampled generalized cross validation function is analogous to the generalization process from cross-validation to GCV provided in \cite{golub1979generalized}.

\section*{Acknowledgments}
This work was partially supported by NSF DMS 1723005 (J. Chung, M. Chung, Tenorio) and NSF DMS 1654175 (J. Chung).

\section*{References}
\bibliographystyle{plain}
\bibliography{references2018}

\begin{thebibliography}{10}

\bibitem{matlabgallery}
Matlab {T}est {M}atrices {G}allery.
\newblock \url{https://www.mathworks.com/help/matlab/ref/gallery.html}.
\newblock Accessed: 2018-11-16.

\bibitem{regtools}
{R}egularization {T}ools {V}ersion 4.1 (for {Matlab}).
\newblock \url{http://www.imm.dtu.dk/~pcha/Regutools/}.
\newblock Accessed: 2018-11-16.

\bibitem{andersen2014generalized}
M~S Andersen and P~C Hansen.
\newblock Generalized row-action methods for tomographic imaging.
\newblock {\em Numerical Algorithms}, 67(1):121--144, 2014.

\bibitem{aster2018parameter}
R~C Aster, B~Borchers, and C~H Thurber.
\newblock {\em Parameter Estimation and Inverse Problems}.
\newblock Elsevier, New York, 2018.

\bibitem{avron2011randomized}
H~Avron and S~Toledo.
\newblock Randomized algorithms for estimating the trace of an implicit
  symmetric positive semi-definite matrix.
\newblock {\em Journal of the ACM}, 58(2):8, 2011.

\bibitem{Bardsley2018}
J~M Bardsley.
\newblock {\em Computational Uncertainty Quantification for Inverse Problems}.
\newblock SIAM, Philadelphia, 2018.

\bibitem{bjorck1996numerical}
A~Bj\"{o}rck.
\newblock {\em Numerical Methods for Least Squares Problems}.
\newblock SIAM, Philadelphia, 1996.

\bibitem{Bottou1998}
L~Bottou.
\newblock {\em Online Learning in Neural Networks}, chapter 2. Online learning
  and stochastic approximations, pages 9--42.
\newblock Cambridge University Press, 1998.

\bibitem{bottou2004large}
L~Bottou and Y~L Cun.
\newblock Large scale online learning.
\newblock In {\em Advances in Neural Information Processing Systems}, pages
  217--224, 2004.

\bibitem{bottou2018optimization}
L~Bottou, F~E Curtis, and J~Nocedal.
\newblock Optimization methods for large-scale machine learning.
\newblock {\em SIAM Review}, 60(2):223--311, 2018.

\bibitem{Byrd2016}
R~H Byrd, S~L Hansen, J~Nocedal, and Y~Singer.
\newblock A stochastic quasi-{N}ewton method for large-scale optimization.
\newblock {\em SIAM Journal on Optimization}, 26(2):1008--1031, 2016.

\bibitem{NASA14}
NASA Goddard Space~Flight Center.
\newblock Image from {NASA}'s lunar reconnaissance orbitor.
\newblock \url{https://lunar.gsfc.nasa.gov/imagesandmultimedia.html}, 2014.

\bibitem{chung2017stochastic}
J~Chung, M~Chung, J~T Slagel, and L~Tenorio.
\newblock Stochastic {N}ewton and quasi-{N}ewton methods for large linear
  least-squares problems.
\newblock {\em arXiv preprint arXiv:1702.07367}, 2017.

\bibitem{chung2006numerical}
J~Chung, E~Haber, and J~Nagy.
\newblock Numerical methods for coupled super-resolution.
\newblock {\em Inverse Problems}, 22(4):1261, 2006.

\bibitem{donoho1995noising}
D~L Donoho.
\newblock De-noising by soft-thresholding.
\newblock {\em IEEE Transactions on Information Theory}, 41(3):613--627, 1995.

\bibitem{drineas2012fast}
P~Drineas, M~Magdon-Ismail, M~W Mahoney, and D~P Woodruff.
\newblock Fast approximation of matrix coherence and statistical leverage.
\newblock {\em Journal of Machine Learning Research}, 13(Dec):3475--3506, 2012.

\bibitem{drineas2011faster}
P~Drineas, M~W Mahoney, S~Muthukrishnan, and T~Sarl{\'o}s.
\newblock Faster least squares approximation.
\newblock {\em Numerische Mathematik}, 117(2):219--249, 2011.

\bibitem{elfving2014semi}
T~Elfving, P~C Hansen, and T~Nikazad.
\newblock Semi-convergence properties of {K}aczmarz’s method.
\newblock {\em Inverse Problems}, 30(5):055007, 2014.

\bibitem{escalante2011alternating}
R~Escalante and M~Raydan.
\newblock {\em Alternating Projection Methods}.
\newblock SIAM, Philadelphia, 2011.

\bibitem{ghattas13}
O~Ghattas.
\newblock {Computational Inverse Problems Can Drive a {B}ig {D}ata Revolution}.
\newblock In A.~M. Bruaset and A.~Tveito, editors, {\em Conversations About
  Challenges in Computing}, pages 43--50. Springer International Publishing,
  New York, 2013.

\bibitem{golub1979generalized}
G~H Golub, M~Heath, and G~Wahba.
\newblock Generalized cross-validation as a method for choosing a good ridge
  parameter.
\newblock {\em Technometrics}, 21(2):215--223, 1979.

\bibitem{haber2012effective}
E~Haber, M~Chung, and F~Herrmann.
\newblock An effective method for parameter estimation with {PDE} constraints
  with multiple right-hand sides.
\newblock {\em SIAM Journal on Optimization}, 22(3):739--757, 2012.

\bibitem{hansen2010discrete}
P~C Hansen.
\newblock {\em Discrete Inverse Problems: Insight and Algorithms}.
\newblock SIAM, 2010.

\bibitem{hansen2006deblurring}
P~C Hansen, J~G Nagy, and D~P O'Leary.
\newblock {\em Deblurring Images: Matrices, Spectra, and Filtering}.
\newblock SIAM, Philadelphia, 2006.

\bibitem{hashemi2015adaptive}
S~Hashemi, S~Beheshti, R~Cobbold, and N~Paul.
\newblock Adaptive updating of regularization parameters.
\newblock {\em Signal Processing}, 113:228--233, 2015.

\bibitem{huang2008three}
B~Huang, W~Wang, M~Bates, and X~Zhuang.
\newblock Three-dimensional super-resolution imaging by stochastic optical
  reconstruction microscopy.
\newblock {\em Science}, 319(5864):810--813, 2008.

\bibitem{kaipio2006statistical}
J~Kaipio and E~Somersalo.
\newblock {\em Statistical and Computational Inverse Problems}, volume 160.
\newblock Springer Science \& Business Media, 2006.

\bibitem{kilmer2001choosing}
M~E Kilmer and D~P O'Leary.
\newblock Choosing regularization parameters in iterative methods for ill-posed
  problems.
\newblock {\em SIAM Journal on Matrix Analysis and Applications},
  22(4):1204--1221, 2001.

\bibitem{Kushner1997}
H~J Kushner and G~G Yin.
\newblock {\em Stochastic Approximation Algorithms and Applications}.
\newblock Springer, New York, 1997.

\bibitem{le2017data}
E~B Le, A~Myers, T~Bui-Thanh, and Q~P Nguyen.
\newblock A data-scalable randomized misfit approach for solving large-scale
  {PDE}-constrained inverse problems.
\newblock {\em Inverse Problems}, 33(6):065003, 2017.

\bibitem{matouvsek2008variants}
J~Matou{\v{s}}ek.
\newblock On variants of the {J}ohnson--{L}indenstrauss lemma.
\newblock {\em Random Structures \& Algorithms}, 33(2):142--156, 2008.

\bibitem{mokhtari2015global}
A~Mokhtari and A~Ribeiro.
\newblock Global convergence of online limited memory {BFGS}.
\newblock {\em The Journal of Machine Learning Research}, 16(1):3151--3181,
  2015.

\bibitem{mueller2012linear}
J~L Mueller and S~Siltanen.
\newblock {\em Linear and Nonlinear Inverse Problems with Practical
  Applications}.
\newblock SIAM, 2012.

\bibitem{Needell2014}
D~Needell and J~A Tropp.
\newblock Paved with good intentions: {A}nalysis of a randomized block
  {K}aczmarz method.
\newblock {\em Linear Algebra and its Applications}, 441:199--221, 2014.

\bibitem{Nocedal2006}
J~Nocedal and S~J Wright.
\newblock {\em Numerical Optimization}.
\newblock Springer, New York, second edition, 2006.

\bibitem{PaSa82b}
C~C Paige and M~A Saunders.
\newblock {{\def\LSQRb{}}Algorithm 583}, {LSQR:} {S}parse linear equations and
  least-squares problems.
\newblock {\em ACM Transactions on Mathematical Software}, 8(2):195--209, 1982.

\bibitem{paige1982lsqr}
C~C Paige and M~A Saunders.
\newblock {LSQR}: An algorithm for sparse linear equations and sparse least
  squares.
\newblock {\em ACM Transactions on Mathematical Software}, 8(1):43--71, 1982.

\bibitem{park2003super}
S~Park, M~Park, and M~Kang.
\newblock Super-resolution image reconstruction: {A} technical overview.
\newblock {\em IEEE signal processing magazine}, 20(3):21--36, 2003.

\bibitem{pilanci2016iterative}
M~Pilanci and M~J Wainwright.
\newblock Iterative {H}essian sketch: Fast and accurate solution approximation
  for constrained least-squares.
\newblock {\em Journal of Machine Learning Research}, 17(53):1--38, 2016.

\bibitem{renaut2017hybrid}
R~A Renaut, S~Vatankhah, and V~E Ardestani.
\newblock Hybrid and iteratively reweighted regularization by unbiased
  predictive risk and weighted {GCV} for projected systems.
\newblock {\em SIAM Journal on Scientific Computing}, 39(2):B221--B243, 2017.

\bibitem{saibaba2017randomized}
A~K Saibaba, Al~Alexanderian, and I~Ipsen.
\newblock Randomized matrix-free trace and log-determinant estimators.
\newblock {\em Numerische Mathematik}, 137(2):353--395, 2017.

\bibitem{shapiro2009lectures}
A~Shapiro, D~Dentcheva, and A~Ruszczy{\'n}ski.
\newblock {\em Lectures on Stochastic Programming: Modeling and Theory}.
\newblock SIAM, Philadelphia, 2009.

\bibitem{strohmer2009randomized}
T~Strohmer and R~Vershynin.
\newblock A randomized {K}aczmarz algorithm with exponential convergence.
\newblock {\em Journal of Fourier Analysis and Applications}, 15(2):262--278,
  2009.

\bibitem{tenorio2017introduction}
L~Tenorio.
\newblock {\em An Introduction to Data Analysis and Uncertainty Quantification
  for Inverse Problems}.
\newblock SIAM, Philadelphia, 2017.

\end{thebibliography}

\end{document}